\documentclass[]{article}

\usepackage[utf8]{inputenc}
\usepackage[english]{babel}
\usepackage[margin=1.25in]{geometry}

\usepackage{amsthm, amsfonts, amsmath, makecell, tikz, 
	graphicx, colortbl}
\usepackage{multirow, booktabs, bigstrut}

\usepackage{mathtools}
\usepackage{adjustbox}

\usepackage{float}
\usepackage{color}
\usepackage{tikz,pgfplots}
\usepackage{tikz-cd}
\usepackage{verbatim}
\usepackage{enumerate}
\usepackage{physics}
\usepackage{braket}

\usepackage{bm}

\usepackage{enumerate, amssymb}

\usetikzlibrary{graphs,graphs.standard,arrows.meta,calc,intersections}

\usepackage[]{algorithm, algorithmic}

\usepackage{tabularx}

\theoremstyle{plain}
\newtheorem{theorem}{Theorem}[section]
\newtheorem{lemma}[theorem]{Lemma}
\newtheorem{corollary}[theorem]{Corollary}
\newtheorem{proposition}[theorem]{Proposition}

\newtheorem{remark}[theorem]{Remark}

\theoremstyle{definition}
\newtheorem{definition}[theorem]{Definition}
\newtheorem{example}[theorem]{Example}

\DeclareMathOperator*{\pmin}{pmin}

\DeclareMathOperator{\Ap}{Ap}

\renewcommand{\epsilon}{\varepsilon}

\newcommand{\R}{{\mathbb R}}
\newcommand{\N}{{\mathbb N}}

\newcommand{\F}{{\mathbb F}}
\newcommand{\Z}{{\mathbb Z}}

\definecolor{gray}{rgb}{.5,.5,.5}

\definecolor{black}{rgb}{0,0,0}

\definecolor{blue}{rgb}{0,0,1}

\definecolor{red}{rgb}{1,0,0}

\definecolor{green}{rgb}{0,1,0}

\definecolor{gold}{rgb}{.5,.5,.2}

\definecolor{yellow}{rgb}{1,1,.4}

\definecolor{purple}{rgb}{.5,0,.5}

\definecolor{darkgreen}{rgb}{0,.5,0}

\definecolor{orange}{rgb}{1,.55,0}

\definecolor{white}{rgb}{1,1,1}

\usepackage[colorinlistoftodos]{todonotes}

\let\originalleft\left
\let\originalright\right
\renewcommand{\left}{\mathopen{}\mathclose\bgroup\originalleft}
\renewcommand{\right}{\aftergroup\egroup\originalright}

\let\originaltodo\todo
\renewcommand{\todo}[1]{\originaltodo[inline]{#1}}

\usepackage{hyperref}
\usepackage{xcolor}
\hypersetup{
	colorlinks,
	linkcolor={red!50!black},
	citecolor={blue!50!black},
	urlcolor={blue!80!black}
}

\usepackage{subfigure}

\title{Extremal factorization lengths of elements in commutative, cancellative semigroups}

\author{Baian Liu}

\begin{document}
	
	\maketitle

\begin{abstract}
	For a numerical semigroup $S \coloneqq \langle n_1, \dots, n_k \rangle$ with minimal generators $n_1 < \cdots < n_k$, Barron, O'Neill, and Pelayo showed that $L(s+n_1) = L(s) + 1$ and $\ell(s+n_k) = \ell(s) + 1$ for all sufficiently large $s \in S$, where $L(s)$ and $\ell(s)$ are the longest and shortest factorization lengths of $s \in S$, respectively. For some numerical semigroups, $L(s+n_1) = L(s) + 1$ for all $s \in S$ or $\ell(s+n_k) = \ell(s) + 1$ for all $s \in S$. In a general commutative, cancellative semigroup $S$, it is also possible to have $L(s+m) = L(s) + 1$ for some atom $m$ and all $s \in S$ or to have $\ell(s+m) = \ell(s) + 1$ for some atom $m$ and all $s \in S$. We determine necessary and sufficient conditions for these two phenomena. We then generalize the notions of Kunz posets and Kunz polytopes. Each integer point on a Kunz polytope corresponds to a commutative, cancellative semigroup. We determine which integer points on a given Kunz polytope correspond to semigroup in which $L(s+m) = L(s) + 1$ for all $s$ and similarly which integer points yield semigroups for which $\ell(s+m) = \ell(s) + 1$ for all $s$. 
\end{abstract}

\section{Introduction}
	\indent\indent In general, semigroups do not have unique factorization. One way to characterize this non-unique factorization is through the use of factorization invariants. Some of these invariants turn out to be eventually quasipolynomial in numerical semigroups, such as in \cite{chap09,chap14,oneill14,oneill17,barron2017set}. In this work, the factorization invariants we are considering are the extremal factorization lengths of elements of a semigroup.

	Let $S$ be a commutative semigroup. We use $+$ for the semigroup operation and we assume that $S$ has an identity denoted as $0$ unless otherwise noted. We say that $s$ is a \textbf{unit} if there exists $-s \in S$ such that $s + -s = 0$. If $0$ is the only unit of $S$, then we say that $S$ is \textbf{reduced}. If we let $U(S)$ denote the group of units of $S$, then we can define $S_{\textup{red}} \coloneqq S/U(S)$. We use the notation $[s]$ for an element $s + U(S) \in S_\textup{red}$. Note that $S_{\textup{red}}$ is reduced. Doing this reduction does not affect factorization lengths of elements of a semigroup, since the factorization lengths of associate elements are the same, so it suffices to consider reduced semigroups for the purposes for factorization lengths. We say that two elements $s, t \in S$ are \textbf{associate}, denoted as $s \sim t$, if there exists some unit $u \in S$ such that $s + u = t$.

	If $S$ is a commutative, cancellative semigroup, then $S$ can be embedded into a group called the \textbf{group of differences} $\mathsf{gp}(S) \coloneqq \{a-b \mid a,b \in S \}$, in which for all $a,b,c,d \in S$, we have $(a-b) + (c-d) = (a+c) - (b+d)$ and $a-b = c-d$ if and only if $a+d = c+b$.

	An important family of commutative, cancellative semigroups is the family of numerical semigroups. A reference for numerical semigroups can be found in \cite{rosales2009numerical}. A \textbf{numerical semigroup} is a semigroup contained in $\N$ with finite complement in $\N$. In this paper, we assume that $0$ is contained in $\N$. Alternatively, we can define a numerical semigroup by taking $n_1, \dots, n_k \in \N\setminus\{0\}$ such that $\gcd(n_1, \dots, n_k) = 1$ and generating the semigroup
	\[
		\langle n_1, \dots, n_k \rangle \coloneqq \{c_1n_1 + \cdots + c_kn_k \mid c_1, \dots, c_k \in \N \}. 
	\]
	We say that $S = \langle n_1, \dots, n_k \rangle$ is \textbf{minimally generated} by $\{ n_1, \dots, n_k \}$ if there is no proper subset of $\{ n_1, \dots, n_k \}$ that generates $S$. More generally, for a set $X$ contained in an abelian group $\Gamma$, we denote by $\langle X \rangle$ the semigroup generated by $X$ in $\Gamma$. Equivalently, this is the smallest semigroup contained in $\Gamma$ containing $X$. 
	
	Now we assume that $S$ is a commutative semigroup. Take $s \in S$ to be a nonzero, nonunit element. We say that $s$ is \textbf{irreducible} or $s$ is an \textbf{atom} if whenever $a+b = s$ for some $a, b \in S$, we have $a$ is a unit or $b$ is a unit. We denote by $\mathcal{A}(S)$ the set of atoms of $S$. For any element $s \in S$, we define the \textbf{set of factorizations} as
	\[
		\mathsf{Z}(s) \coloneqq \left\{(c_{[a]}) \in \N^{\mathcal{A}(S_\textup{red})}\,\middle\vert\, \sum_{[a] \in \mathcal{A}(S_\textup{red})} c_{[a]} [a] = [s] \right\}.
	\]
	For convenience, we write $(c_a)$ for $(c_{[a]})$. In this paper, we also use the notation $\N^X$ to denote the direct sum of copies of $\N$ indexed by a set $X$. Note that since $(c_a)$ has only finitely many nonzero elements, the sum $\sum_{[a] \in \mathcal{A}(S_\textup{red})} c_a [a]$ is a finite sum. The set of factorizations can be thought of as all the ways $s$ can be written as the sum of atoms, up to association. Alternatively, we can first define the \textbf{factorization homomorphism} of $S$ as the function $\varphi_S: \N^{\mathcal{A}(S_\textup{red})} \to S_\textup{red}$ defined by 
	\[
		\varphi_S(c_a) = \sum_{[a] \in \mathcal{A}(S_\textup{red})} c_a [a] 
	\] 
	and define the set of factorizations as $\mathsf{Z}(s) = \varphi_S^{-1}([s])$. For convenience, we also define $\varphi_S(c_a) = \sum_{[a] \in X} c_a [a]$ if $(c_a) \in \N^X$ and $X$ is a subset of $\mathcal{A}(S_\textup{red})$, and since the set of factorizations does not depend on the representatives of the elements of $S_\textup{red}$, we take $\varphi_S(c_a)$ to also mean any element $s \in S$ such that $[s] = \varphi_S(c_a)$. Additionally, we define $e_a$ to be element of $\N^X$ that is $1$ at the coordinate indexed by $[a]$ and $0$ everywhere else. We also consider $\N^X$ to be equipped with the usual coordinatewise addition and subtraction. Lastly, we put a partial order on $\N^X$ where $(c_a) \leq (c_a')$ if and only if $c_a \leq c_a'$ for all $[a] \in X$. 
	
	One way we can analyze the set of factorizations is to look at the length of each factorization. More generally, if we have an element $(c_\lambda)_{\lambda \in \Lambda} \in \N^\Lambda$, we can define the \textbf{length} of $(c_\lambda)$ as $\abs{(c_\lambda)} = \sum_{\lambda \in \Lambda} c_\lambda$. Note that this is again a finite sum. Using this, for an element $s \in S$, we can define its \textbf{set of lengths} to be
	\[
		\mathcal{L}(s) \coloneqq \left\{\abs{(c_a)} \,\middle\vert\, (c_a) \in \mathsf{Z}(s) \right\}. 
	\]
	If $\mathcal{L}(s)$ is not empty, or equivalently if $\mathsf{Z}(s)$ is not empty, we can derive two important invariants from this set, namely
	\[
		L(s) \coloneqq \sup \mathcal{L}(s) \quad \text{and} \quad \ell(s) \coloneqq \min \mathcal{L}(s),
	\]
	the \textbf{longest factorization length} and \textbf{shortest factorization length} of $s$ in $S$, respectively. For the factorization invariants, we may use a subscript to emphasize the semigroup to which the element belongs. For example, $\mathsf{Z}_S(s), L_S(s),$ and $\ell_S(s)$. The longest and shortest factorization lengths in a numerical semigroup behave nicely according to the following result from \cite{barron2017set}.

	\begin{theorem}\cite[Theorems 4.2 and 4.3]{barron2017set}\label{Thm:EventualQuasilinearity}
		Let $\langle n_1, \dots, n_k \rangle$ be a numerical semigroup minimally generated by $n_1, \dots, n_k \in \N$ such that $n_1 < \cdots < n_k$. The following statements hold.
		\begin{enumerate}
			\item 	If $n > (n_1-1)n_k-n_1$, then $L(n+n_1) = L(n) + 1$.
			\item 	If $n > (n_{k}-1)n_{k-1} - n_k$, then $\ell(n+n_k) = \ell(n) + 1$.
		\end{enumerate}
	\end{theorem}
	
	It is possible that for a numerical semigroup $S = \langle n_1, \dots, n_k \rangle$ minimally generated by $n_1, \dots, n_k \in \N$ such that $n_1 < \cdots < n_k$ to have $L(n+n_1) = L(n) + 1$ hold for all $n \in S$. Similarly, it is possible that $\ell(n+n_k) = \ell(n) + 1$ for all $n \in S$. For example, for $S = \langle 6, 9, 20 \rangle$, we have $L(n+6) = L(n) + 1$ and $\ell(n+20) = \ell(n) + 1$ for all $n \in S$ \cite[Theorems 4.1 and 4.2]{chapman2018factoring}. 
	
	The question of when $L(n+n_1) = L(n) + 1$ for all $n \in S$ and when $\ell(n+n_k) = \ell(n) + 1$ for all $n \in S$ for a numerical semigroup $S$ was posed in Question 4.5 in \cite{chapman2018factoring}. This question was answered using Betty elements for the case of numerical semigroups of embedding dimension three in \cite{LiuYap}. We generalize this question further. Take a commutative, cancellative semigroup $S$. We want to determine for each atom $m \in \mathcal{A}(S)$ whether $L(s+m) = L(s) + 1$ for all $s \in S$ and whether $\ell(s+m) = \ell(s) + 1$ for all $s \in S$. One of the ways we will characterize when $L(s+m) = L(s) + 1$ for all $s \in S$ and when $\ell(s+m) = \ell(s) + 1$ is through generalizing Kunz posets and Kunz polytopes. 
	
	Kunz posets and Kunz polytopes have been used recently to study numerical semigroups \cite{WilfConj, Kunz1, Kunz2, Kunz3}. The advantage of this is that it allows for families of numerical semigroups with similar properties to be studied at once. These tools have been used to verify Wilf's Conjecture for numerical semigroups of small multiplicity, study the oversemigroups of a numerical semigroup, and describe the minimal presentations of a numerical semigroup. We extend these tools to the setting of general commutative, cancellative semigroups in order to study the longest and shortests factorization lengths of semigroup elements. 	
	
	Since we will be utilizing posets, and more generally preordered sets, we set up some definitions and notation surrounding these ideas. Let $X$ be a set. A \textbf{preorder} on $X$ is a binary relation $\sqsubseteq$ such that $a \sqsubseteq a$ for all $a \in X$ and for all $a, b, c \in X$, if $a \sqsubseteq b$ and $b \sqsubseteq c$, then $a \sqsubseteq c$. The set $X$ along with a preorder on $X$ is a \textbf{preordered set}. Note that a partial order is a preorder and a poset is a preordered set.

 	We also use the notation 
 	\[
		\min\limits_{\leq} X \coloneqq \{ x \in X \mid \not\exists\, y \in X \setminus\{x\} \text{ s. t. }y \sqsubseteq x  \}
	\] to denote the set of minimal elements of $X$ under a partial order $\leq$, and
	\[
		\pmin\limits_{\sqsubseteq} X \coloneqq \{x \in X \mid \forall\, y \in X\,( y \sqsubseteq x \implies x \sqsubseteq y) \}
	\]
	to denote the set of \textbf{pseudominimal} elements of $X$ under a preorder $\sqsubseteq$. 
	
	Let $S$ be a commutative, cancellative semigroup and $m \in \mathcal{A}(S)$. In this paper, in Section \ref{Sect:Exceptions} we first find potential exceptions $s \in S$ to the formula $L(s+m) = L(s) + 1$ that will indicate whether if $S$ has any exceptions to the formula. We also show that there is an analogous set for detecting exceptions to the $\ell(s+m) = \ell(s) + 1$ formula. Furthermore, we give results to reduce the size of the set of potential exceptions. In Section \ref{Sect:Kunz}, we generalize Kunz posets and Kunz polytopes to analyze the longest and shortest length formulas for an entire family of semigroups at once. This allows us to determine inequalities that dictate which semigroups in the family have exceptions to the longest and shortest length formulas and which ones do not. 
	
\section{Possible exceptions}\label{Sect:Exceptions}

	\indent\indent According to Theorem \ref{Thm:EventualQuasilinearity}, for a numerical semigroup $S$ minimally generated by $n_1, \dots, n_k \in \N$ with $n_1 < \cdots < n_k$, we have $L(s + n_1) = L(s) + 1$ for sufficiently large $s \in S$ and $\ell(s+ n_k) = \ell(s) + 1$ for sufficiently large $s \in S$. To see if these formulas hold for all $s \in S$, there are only a finite number of possible exceptions to check. This set of potential exceptions can still be quite large. It turns out not all of the small elements need to be checked. For example, for a numerical semigroup generated by three elements, it suffices to check the validity of the formulas at only a single element according to the following theorem.
	\begin{theorem}\cite[Theorem 3.3]{LiuYap}\label{Thm:LiuYap}
		Let $S = \langle n_1, n_2, n_3 \rangle$ be a numerical semigroup minimally generated by $n_1, n_2, n_3 \in \N$ with $n_1 < n_2 < n_3$. Then 
		\begin{enumerate}
			\item $L(s + n_1) = L(s) + 1$ for all $s \in S$ if and only if $L(\alpha(S)n_2) = L(\alpha(S)n_2 - n_1) + 1$ and
			\item $\ell(s+n_3) = \ell(s) + 1$ for all $s \in S$ if and only if $\ell(\beta(S)n_2) = \ell(\beta(S)n_2-n_3) + 1$, 
		\end{enumerate}
		where $\alpha(S) = \min\{c \in \N \mid cn_2 - n_1 \in S \}$ and $\beta(S) = \min\{c \in \N \mid cn_2 - n_3 \in S\}$. 
	\end{theorem}

	We will show that for a general commutative, cancellative semigroup $S$ such that $L(s)$ is finite for all $s \in S$, to check that for some $m \in \mathcal{A}(S)$ that $L(s+m) = L(s) + 1$ for all $s \in S$, it suffices to verify the formula for all $s$ in some proper subset of $S$. The analogous statement holds for $\ell(s+m) = \ell(s) + 1$ for some fixed $m \in \mathcal{A}(S)$. If $S$ is a numerical semigroup, this proper subset of $S$ will be much smaller than the set of semigroup elements no more than the thresholds given in Theorem \ref{Thm:EventualQuasilinearity}. To arrive at a smaller set of potential exceptions, we make use of the following lemma. This is a generalization of Lemma 3.1 of \cite{LiuYap}. 
	
	\begin{lemma}\label{Lem:Noms}
		Let $S$ be a commutative, cancellative semigroup and take $s \in S$ and $m \in \mathcal{A}(S)$. Suppose that $s-m \in S$ and $L(s)$ is finite. Then
		\begin{enumerate}
			\item there exists $(c_a) \in \mathsf{Z}(s)$ such that $\abs{(c_a)} = L(s)$ with $c_m > 0$ if and only if $L(s) = L(s-m) + 1$, and
			\item there exists $(c_a) \in \mathsf{Z}(s)$ such that $\abs{(c_a)} = \ell(s)$ with $c_m > 0$ if and only if  $\ell(s) = \ell(s-m) + 1$. 
		\end{enumerate}
	\end{lemma}
	
	\begin{proof}
		Let $s \in S$ and suppose there exists $(c_a) \in \mathsf{Z}(s)$ such that $\abs{(c_a)} = L(s)$ with $c_m > 0$. Then $(c_a) - e_m \in \mathsf{Z}(s-m)$. This implies that \[L(s-m) \geq \abs{(c_a) - e_m} = \abs{(c_a)} - 1 = L(s) - 1.\]
		On the other hand, take $(b_a) \in \mathsf{Z}(s-m)$. Then $(b_a) + e_m \in \mathsf{Z}(s)$. This yields $L(s) \geq \abs{(b_a) + e_m} = \abs{(b_a)} + 1$, which means that
		\[
		 	L(s) \geq L(s-m) + 1.
		\]
		Combining the two inequalities $L(s-m) \geq L(s) - 1$ and $L(s) \geq L(s-m) + 1$ gives us that $L(s) = L(s-m) + 1$.
		
		Conversely, suppose that $L(s) = L(s - m) + 1$. Then take $(b_a) \in \mathsf{Z}(s-m)$ such that $\abs{(b_a)} = L(s-m)$. Now we have $(b_a) + e_m \in \mathsf{Z}(s)$. Since $\abs{(b_a) + e_m} = L(s-m) + 1 = L(s)$ and the coordinate indexed by $[m]$ in $(b_a) + e_m$ is strictly positive, the converse has been proven. 
		
		The proof of the statement concerning the shortest factorization length is similar. 
	\end{proof}
	
	If we fixed an $m \in \mathcal{A}(S)$ for a commutative, cancellative semigroup, then according to the previous lemma, for any exception to $L(s+m) = L(s) + 1$ where $s \in S$, any factorization of $s+m$ maximal length cannot an atom associate to $m$. Therefore, we should consider factorizations $(c_a)$ in $\mathsf{Z}(s+m)$ such that $c_m = 0$. Furthermore, since $(s+m)- m \in S$ and thus $[s+m] - [m] \in S_\textup{red}$, we want to consider sums of atoms of $S_\textup{red}$ that are not $[m]$ such that subtracting $[m]$ keeps us in $S_\textup{red}$. Furthermore, there exist sums of elements of $\mathcal{A}(S_{\textup{red}}) \setminus \{[m]\}$ that are in a sense minimal with respect to the property that subtracting $[m]$ yields an element of $S_\textup{red}$. An analogous line of thinking shows that the same idea can also be used to analyze $\ell(s+m) = \ell(s) + 1$.
	
	\begin{definition}	
	Let $S$ be a commutative, cancellative semigroup and let $m \in \mathcal{A}(S)$. We define
	\[
		\mathsf{Repl}_m(S) \coloneqq \{(c_a) \in \N^{\mathcal{A}(S_\textup{red}) \setminus \{[m]\}} | -m + \varphi_S(c_a) \in S \}.
	\]
	Note that the fact that $-m + \varphi_S(c_a) \in S$ does not depend on the representative of $\varphi_S(c_a)$ we choose. We can think of $\mathsf{Repl}_m(S)$ as factorizations without atoms associate to $m$ that can be replaced with another factorization of the same element that involves the atom $m$. To consider factorizations that are minimal with respect to this property, we define
	\[
		\mathsf{MinRepl}_m(S) \coloneqq \min\limits_{\leq} \mathsf{Repl}_m(S).
	\]
	Note that under the partial order on $\N^{\mathcal{A}(S_\textup{red}) \setminus \{[m]\}}$, for every $(c_a) \in \mathsf{Repl}_m(S)$, there is some $(c_a') \in \mathsf{MinRepl}_m(S)$ such that $(c_a') \leq (c_a)$. 
	\end{definition}
	
	\begin{lemma}\label{Lem:Minimalize}
		Let $S$ be a commutative, cancellative semigroup. Take $s \in S$ and $m \in \mathcal{A}(S)$ such that $s-m \in S$ and $L(s)$ is finite. Then
		\begin{enumerate}
			\item if $L(s) \neq L(s-m) + 1$, then for all $(c_a) \in \N^{\mathcal{A}(S_\textup{red}) \setminus \{[m]\}}$ such that $s \sim \varphi_S(c_a)$ and $L(s) = \abs{(c_a)}$ and for every $(c_a') \in \mathsf{MinRepl}_m(S)$ such that $(c_a') \leq (c_a)$, we have $L\left(\varphi_S(c_a')\right) \neq L\left(-m + \varphi_S(c_a') \right) + 1$, and 
			\item if $\ell(s) \neq \ell(s-m) + 1$, then for all $(c_a) \in \N^{\mathcal{A}(S_\textup{red}) \setminus \{[m]\}}$ such that $s \sim \varphi_S(c_a)$ and $\ell(s) = \abs{(c_a)}$ and for every $(c_a') \in \mathsf{MinRepl}_m(S)$ such that $(c_a') \leq (c_a)$, we have $\ell\left(\varphi_S(c_a')\right) \neq \ell\left(-m + \varphi_S(c_a')\right) + 1$. 
		\end{enumerate}
	\end{lemma}
	
	\begin{proof}
		Suppose that $L(s) \neq L(s - m) + 1$. Let $(c_a) \in \N^{\mathcal{A}(S_\textup{red}) \setminus \{[m]\}}$ such that $s \sim \varphi_S(c_a)$ and $L(s) = \abs{(c_a)}$. Since $s - m \in S$, we have that $(c_a) \in \mathsf{Repl}_m(S)$. Now take some $(c_a') \in \mathsf{MinRepl}_m(S)$ such that $(c_a') \leq (c_a)$. Let $(b_a) \in \mathsf{Z}\left(-m + \varphi_S(c_a')\right)$ such that $\abs{(b_a)} = L\left(-m +\varphi_S(c_a')\right)$. Then
		\[
		(b_m+1)[m] + \sum_{[a] \in \mathcal{A}(S_\textup{red}) \setminus \{[m]\}} b_a [a] = \sum_{[a] \in \mathcal{A}(S_\textup{red}) \setminus \{[m]\}} c_a' [a].
		\]
		Adding $\sum_{[a] \in \mathcal{A}(S_\textup{red}) \setminus \{[m]\}}(c_a - c_a') [a]$ to both sides gives us
		\[
		(b_m+1)[m] + \sum_{[a] \in \mathcal{A}(S_\textup{red}) \setminus \{[m]\}} (b_a + c_a - c_a') [a] = \sum_{[a] \in \mathcal{A}(S_\textup{red}) \setminus \{[m]\}} c_a [a]= [s].
		\]
		Since the left hand side is a factorization of $s$ with at least one $m$, we must have
		\[
		\abs{(b_a)} + 1 + \abs{(c_a)} - \abs{(c_a')} < L(s+m) = \abs{(c_a)}
		\]
		by Lemma \ref{Lem:Noms}. Thus, we must have
		\[
		L\left(-m + \varphi_S(c_a') \right) + 1 = \abs{(b_a)} + 1 < \abs{(c_a')} \leq L\left(\varphi_S(c_a')\right),
		\]
		showing that $L\left(\varphi_S(c_a')\right) \neq L\left(-m + \varphi_S(c_a')\right) + 1$. The proof for the second statement is similar.
	\end{proof}
	
	The following theorem shows that to find exceptions to $L(s+m) = L(s) + 1$ and $\ell(s+m) = \ell(s) + 1$, it suffices to look at elements $s + m$ such that $s+m \sim \varphi_S(c_a)$ for some $(c_a) \in \mathsf{MinRepl}_m(S)$. This does not, however, imply that all the exceptions are of this form. 
	
	\begin{theorem}\label{Thm:PotentialExceptions}
		Let $S$ be a commutative, cancellative semigroup such that $L(s)$ is finite for all $s \in S$. Then for $m \in \mathcal{A}(S)$, 
		\begin{enumerate}
			\item we have $L(s+m) = L(s) + 1$ for all $s \in S$ if and only if $L(\varphi_S(c_a)) = L(-m + \varphi_S(c_a)) + 1$ for all $(c_a) \in \mathsf{MinRepl}_m(S)$, and
			\item we have $\ell(s+m) = \ell(s) + 1$ for all $s \in S$ if and only if  $\ell(\varphi_S(c_a)) = \ell(-m + \varphi_S(c_a)) + 1$ for all $(c_a) \in \mathsf{MinRepl}_m(S)$.
		\end{enumerate}
	\end{theorem}
	
	\begin{proof}
		The forward directions of both statements are immediate.

		Now fix $m \in \mathcal{A}(S)$. Suppose that there exists $s \in S$ such that $L(s+m) \neq L(s) + 1$. By Lemma \ref{Lem:Noms}, for every $(c_a) \in \mathsf{Z}(s+m)$ such that $\abs{(c_a)} = L(s+m)$, we must have $c_m = 0$. Fix some $(c_a) \in \mathsf{Z}(s+m)$ such that $\abs{(c_a)} = L(s+m)$. Since $s+m - m \in S$ and $c_m = 0$, we have that $\widehat{(c_a)} \in \mathsf{Repl}_m(S)$, where $\widehat{(c_a)}$ is $(c_a)$ with the entry indexed by $[m]$ removed. Then there must be some $(c_a') \in \mathsf{MinRepl}_m(S)$ such that $(c_a') \leq \widehat{(c_a)}$. By Lemma \ref{Lem:Minimalize}, we know that $L(\varphi_S(c_a')) \neq L(-m + \varphi_(c_a')) + 1$ for some $(c_a') \in \mathsf{MinRepl}_m(S)$. The proof for the second statement is similar.
	\end{proof}
	
	We may not even need to check every element of $\mathsf{MinRepl}_m(S)$. It is possible that for distinct $(c_a), (c_a') \in \mathsf{MinRepl}_m(S)$, we have $\varphi_S(c_a) \sim \varphi_S(c_a')$. 
	
	For the longest factorization length, we always have $L(\varphi_S(c_a)) \geq L(-m + \varphi_S(c_a)) + 1$. For $L(\varphi_S(c_a)) > L(-m + \varphi_S(c_a)) + 1$ to happen, according to Lemma \ref{Lem:Noms}, it must be the case that for every $(b_a) \in \mathsf{Z}(\varphi_S(c_a))$ such that $\abs{(b_a)} = L(\varphi_S(c_a))$, we have $b_m = 0$. If every such factorization is length $2$, then for every factorization $(b_a') \in \mathsf{Z}(\varphi_S(c_a))$ with $b_m' > 0$, we must have $(b_a') = e_m$. This forces $\varphi_S(c_a) \sim m$, which is a contradiction since $m$ is an atom of $S$ and therefore cannot be written as the sum of atoms of $S$ that are not associates of $m$. Thus, if $s \in S$ is such that whenever $(c_a) \in \mathsf{MinRepl}_m(S)$ with the property $\varphi_S(c_a) \sim s$, we have $\abs{(c_a)} = 2$, then we must have $L(\varphi_S(c_a)) = L(-m + \varphi_S(c_a)) + 1$. The fact has no analogue for the shortest factorization length since the inequality that always holds true, $\ell(\varphi_S(c_a)) \leq \ell(-m + \varphi_S(c_a)) + 1$, is reversed.
	
	Furthermore, we use divisibility to potentially reduce the set of potential exceptions even further in the following corollary. For elements $s$ and $t$ in a commutative, cancellative semigroup $S$, we say that $s|t$, or $s$ \textbf{divides} $t$, if there exists some $s'\in t$ such that $s + s' = t$. We note that additionally, divisibility gives a preorder on $S$ and a partial order on $S$ if $S$ is reduced.  

	\begin{corollary}\label{Cor:Reduction}
		Let $S$ be a commutative, cancellative semigroup such that $L(s)$ is finite for all $s \in S$. Take $m \in \mathcal{A}(S)$ and set, as subset of $S_\textup{red}$, 
		\[
			M_m^{(1)}(S) \coloneqq \left\{ [s] \in \min\limits_|\{\varphi_S(c_a) \mid (c_a) \in  \mathsf{MinRepl}_m(S) \} \,\middle\vert\, \stackbox[l][l]{$\exists\, (c_a) \in \mathsf{MinRepl}_m(S)$ such that  \\ $\abs{(c_a)} > 2, \varphi_S(c_a) = [s]$ } \right\}
		\]
			and
		\[ 
			M_m^{(2)}(S) \coloneqq \min\limits_|\left\{ \varphi_S(c_a) \,\middle\vert\, (c_a) \in \mathsf{MinRepl}_m(S) \right\}.
		\]
		Then for $m \in \mathcal{A}(S)$, 
		\begin{enumerate}
			\item we have $L(s+m) = L(s) + 1$ for all $s \in S$ if and only if $L(s) = L(s-m) + 1$ for all $[s] \in M_m^{(1)}(S)$, and
			\item we have $\ell(s+m) = \ell(s) + 1$ for all $s \in S$ if and only if $\ell(s) = \ell(s-m) + 1$ for all $[s] \in M_m^{(2)}(S)$.
		\end{enumerate}
	\end{corollary}
	
	\begin{proof}
		If we replaced $M_m^{(1)}(S)$ and $M_m^{(2)}(S)$ with $\left\{ \varphi_S(c_a) \,\middle\vert\, (c_a) \in \mathsf{MinRepl}_m(S) \right\}$ in the statement, then we have true statements due to Theorem \ref{Thm:PotentialExceptions}. Now it remains to show that it is sufficient to only consider the minimal elements under divisibility. For the longest factorization length, we also have the lengths of sequences in $\mathsf{MinRepl}_m(S)$ that realize these minimal elements to consider.
		
		Suppose that for each $[s] \in M_m^{(1)}(S)$, we have that $L(s) = L(s-m) + 1$. Now suppose for a contradiction that there is some $(b_a) \in \mathsf{MinRepl}_m(S)$ such that $L\left(\varphi_S(b_a) \right)\neq L\left(-m + \varphi_S(b_a) \right) + 1$. Then there exists some $[s] \in \min\limits_|\{\varphi_S(c_a) \mid (c_a) \in  \mathsf{MinRepl}_m(S) \}$ such that $[s]$ divides $\varphi_S(b_a)$ in $S_\textup{red}$. Either $[s] \in M_m^{(1)}(S)$ or for all $(B_a) \in \mathsf{MinRepl}_m(S)$ such that $\varphi_S(B_a) = [s]$, we have $\abs{(B_a)} = 2$. We break into these two cases.
		
		\begin{enumerate}[i.]
			\item If $[s] \in M_m^{(1)}(S)$, then there exists some $(B_a) \in \mathsf{MinRepl}_m(S)$ such that $\abs{(B_a)} > 2$ and $\varphi_S(B_a) = s$. Since $[s] | \varphi_S(b_a)$, there exists $(B_a') \in \mathsf{Repl}_m(S)$ such that $(B_a) \leq (B_a')$ and $\varphi_S(B_a') = \varphi_S(b_a)$. However, since $L(\varphi_S(B_a')) \neq L(-m + \varphi_S(B_a')) + 1$ by assumption, we get that $L(s) \neq L(-m + s) + 1$ by Lemma \ref{Lem:Minimalize}, which is a contradiction. 
			
			\item Suppose for all $(B_a) \in \mathsf{MinRepl}_m(S)$ such that $\varphi_S(B_a) = [s]$, we have $\abs{(B_a)} = 2$. If a factorization of $s$ of length $L(s)$ can be realized by a factorization with at least one $m$, then $L(s) = L(s-m) + 1$, meaning that $L(\varphi_S(c_a)) = L(-m + \varphi_S(c_a)) + 1$ by Lemma \ref{Lem:Minimalize}, a contradiction. Thus, factorizations of $s$ of length $L(s)$ can only be realized by factorizations that do not contain $m$. Therefore, $L(s) = \max\{\abs{(c_a)} \mid (c_a) \in \mathsf{Repl}_m(S), \varphi_S(c_a) = [s] \}$. Since $[s]$ is minimal with respect to divisibility in the set $\{ \varphi_S(c_a) \mid (c_a) \in \mathsf{MinRepl}_m(S)\}$, we must have that each $(c_a) \in \mathsf{Repl}_m(S)$ such that $\varphi_S(c_a) = [s]$ is actually also in $\mathsf{MinRepl}_m(S)$. This implies that $L(s) = 2$. Together with the fact that $L(s) \geq L(-m + s) + 1$, this shows that $L(s) = L(-m + s) + 1$ only fails if $L(-m + s) = 0$, meaning $s \sim m$. However, this is a contradiction since $s \sim \varphi_S(b_a)$ implies $s$ is a sum of atoms of $S$ that are not associates of $m$.  
		\end{enumerate}
		
		Since both cases reach contradictions, we know that $L(\varphi_S(b_a)) = L(-m + \varphi_S(b_a)) + 1$ for all $(b_a) \in \mathsf{MinRepl}_m(S)$. The proof for the second statement is similar, but we do not have to consider the lengths of sequences in $\mathsf{MinRepl}_m(S)$ that evaluate to elements of $M_m^{(2)}(S)$ under the factorization homomorphism. 
	\end{proof}
	
	Now we take a look at examples of using the set $\mathsf{MinRepl}_m(S)$. We save an example with a numerical semigroup for later in the section since we will develop more tools specifically for numerical semigroups.

	\begin{example}
		Consider $\Z^2$ with coordinate-wise addition and the subsemigroup
		\[
			S \coloneqq \langle (3,0), (7,0), (11,0), (6,1), (0,3) \rangle.
		\]
		Note that $\mathcal{A}(S) = \{(3,0), (7,0), (11,0), (6,1), (0,3)\}$. Since $S$ is reduced, we may drop the square brackets from the notation.  
		
		We can calculate that $\mathsf{MinRepl}_{(3,0)}(S) = \{(3,0,0,0), (0,3,0,0), (0,0,3,0), (1,1,0,0)\}$ with the coordinates indexed by $(7,0), (11,0), (6,1), (0,3)$. Evaluating the elements of $\mathsf{MinRepl}_{(3,0)}(S) $ using $\varphi_S$ gives $(21,0), (33,0), (18,3), (18,0)$. We see that $(18,0)|(21,0)|(33,0)$ and $(18,0)|(18,3)$. Furthermore, only $(1,1,0,0) \in \mathsf{MinRepl}_{(3,0)}$, which is of length $2$, corresponds to the element $(18,0)$, so $M_m^{(1)}(S) = \emptyset$. Thus, $L(s + (3,0)) = L(s) + 1$ for all $s \in S$. 
	\end{example}
	
	\begin{example}
		Consider $\Z^2$ with coordinate-wise addition and the semigroup
		\[
			S \coloneqq \langle (2,0), (3,1), (0,5) \rangle.
		\]
		Note that $\mathcal{A}(S) = \{ (2,0), (3,1), (0,5)\}$. Again, $S$ is reduced, so we drop the square brackets from the notation.  
		
		We calculate that $\mathsf{MinRepl}_{(2,0)}(S) = \{(10,0)\}$, where the coordinates are indexed by $(3,1), (0,5)$. In order to use Theorem \ref{Thm:PotentialExceptions}, we calculate $L((30,10)) = 17$, realized only by $15(2,0) + 2(0,5)$, which has at least one $(2,0)$. Thus, by Lemma \ref{Lem:Noms}, we know that $L((30, 10)) = L((28, 10)) + 1$. Then, by Theorem \ref{Thm:PotentialExceptions}, we know that $L(s + (2,0)) = L(s) + 1$ for all $s \in S$. 
		
		Next, we calculate that $\mathsf{MinRepl}_{(3,1)}(S) = \{(15,2)\}$, in which the coordinates are indexed by $(2,0), (0,5)$. We know that $15(2,0) + 2(0,5) = (30,10)$. However, we have seen that $L((30,10)) = 17$ and the only factorization of $(30,10)$ of length $17$ does not contain a $(3,1)$. Thus, $L((30,10)) \neq L((27,9)) + 1$ by Lemma \ref{Lem:Noms}. Indeed, $L((27,9)) = 9$, realized by $9(3,1) = (27,9)$. 
		
		Lastly, we consider $(0,5)$. We find that $\mathsf{MinRepl}_{(0,5)}(S) = \{(0,10)\}$, where the coordinates are indexed by $(2,0), (3,1)$. Again, $L((30,10)) = 17$ and the only factorization of $(30,10)$ of length $17$ does contain a $(0,5)$, so $L((30,10)) = L((30,5)) + 1$ by Lemma \ref{Lem:Noms} and thus $L(s+(0,5)) = L(s) + 1$ for all $s \in S$. 
		
		To summarize the results for the longest factorization lengths, we have $L(s+(2,0)) = L(s+(0,5)) = L(s) + 1$ for all $s \in S$. 
		
		As for the shortest factorization lengths, each of the $\mathsf{MinRepl}_m(S)$ sets for $m \in \mathcal{A}(S)$ has a singleton which evaluates to $(30, 10)$ under $\varphi_S$.  We see that $\ell((30,10)) = 10$ and the only factorization of $(30,10)$ of length $10$ is $10(3,1)$, which contains only $(3,1)$. Thus, for $m \in \mathcal{A}(S)$, we have $\ell(s+m) = \ell(s) + 1$ for all $s \in S$ if and only if $m = (3,1)$ combining Lemma \ref{Lem:Noms} and Theorem \ref{Thm:PotentialExceptions}. 
	\end{example}
	
	The following example considers the multiplicative semigroup of a semigroup algebra. 
	
	\begin{example}
		Let $S$ be the nonzero elements of $\F_2[x^2, x^3] = \{a_0 + a_1x + a_2x^2 + \cdots + a_nx^n \in \F_2[x] \mid a_1 = 0 \}$ with multiplication as the semigroup operation. We see that $S$ is reduced so we can drop the square brackets that indicate elements of $S_\textup{red}$. We want to show that $L(s\cdot x^2) = L(s) + 1$ for all $s \in S$. 
		
		Let $(c_a) \in \mathsf{MinRepl}_{x^2}(S)$. We know that \[x^2 \cdot f = \prod\limits_{a \in \mathcal{A}(S) \setminus \{x^2\}} a^{c_a}\] for some $f \in S$. Since $\F_2[x^2, x^3]$ is contained in the unique factorization domain $\F_2[x]$, we know that $x$ divides some $g \in \mathcal{A}(S)\setminus\{x^2\}$ in $\F_2[x]$ for which $c_g > 0$. For such an element $g$, we have $g \in S$, so we must have $x^2$ divide $g$ in $\F_2[x]$. Suppose $x^2$ divides $g$ in $\F_2[x]$. We write $g = x^k \cdot h$, where $k \in \N$ with $k\geq 2$ and $h \in \F_2[x]$ with the constant term of $h$ being $1$. Furthermore, the coefficient of $x$ in $h$ is $1$ since otherwise $h \in S$, implying that $g \notin \mathcal{A}(S)$. If $k \geq 4$, then $g = x^2 \cdot x^{k-2}h$ contradicts the fact that $g \in \mathcal{A}(S)$, so $k \in \{2, 3\}$. We temporarily suppose that for all $a \in \mathcal{A}(S)\setminus\{x^2\}$ such that $c_a > 0$, we have that $x$ does not divide $a$ in $\F_2[x]$. Then we calculate
		\[
			f = x^{k-2} h g^{c_g-1} \prod_{a \in \mathcal{A}(S) \setminus\{x^2, g\}} a^{c_a}.
		\]
		Our temporary assumption forces $c_g > 1$. Suppose that $c_g = 1$. Then $k \in \{2, 3\}$ forces the coefficient of $x$ in $f$ to be $1$, contradicting the fact that $f \in S$. Therefore, we must have $c_g > 1$ or there exists some $G \in \mathcal{A}(S) \setminus\{x^2, g\}$ with $c_G > 0$ and $x$ dividing $G$ in $\F_2[x]$.  In both cases, we have $a_1, a_2 \in \mathcal{A}(S) \setminus \{x^2\}$ which may or may not be distinct, such that $(c_a) - e_{a_1} - e_{a_2} \in \N^{\mathcal{A}(S)\setminus\{m\}}$ and $a_i = x^{k_i} \cdot A_i$, where $k_i \in \{2,3\}$ and $A_i \in \F_2[x]$ has the coefficient $1$ for both the constant term and the degree one term for $i \in \{1,2\}$. Then, 
		\[
			a_1a_2 = x^{k_1}(1+x+A_1')x^{k_2}(1+x+A_2') = x^{k_1+k_2}(1+x^2 + A_1'+A_1'x+A_2'+A_2'x+A_1'A_2'),
		\]
		where $A_i' = 1+x+A_i$ and thus $x^2$ divides $A_i'$ in $\F_2[x]$ for each $i \in \{1,2\}$, implying that $1+x^2 + A_1'+A_1'x+A_2'+A_2'x+A_1'A_2' \in S$. Additionally, $k_1 +k_2 - 2 \geq 2$, so $x^{k_1+k_2} \in S$ as well, meaning $\frac{a_1a_2}{x^2} \in S$. Since $(c_a)$ is minimal in $\mathsf{Repl}_{x^2}(S)$, we know $(c_a) = e_{a_1} + e_{a_2}$. This shows that $M_m^{(1)}(S) = \emptyset$. By Corollary \ref{Cor:Reduction}, we have $L(s\cdot x^2) = L(s) + 1$ for all $s \in S$.
	\end{example}
	
	For the remainder of this section, we shift our focus to numerical semigroups. Since numerical semigroups are reduced, we again drop the square brackets that indicate an equivalence class of associate elements. We can reduce the set of potential exceptions to $L(s+m) = L(s+1)$ and $\ell(s+m) = \ell(s) + 1$ further in the case of numerical semigroups. However, first, we show that there is only one atom $m$ for which these formulas have the potential to hold for all elements of the semigroup. 
	
	\begin{proposition}
		Let $S \coloneqq \langle n_1, \dots, n_k \rangle$ be a numerical semigroup minimally generated by the elements $n_1, \dots, n_k \in \N$ with $n_1 < \cdots < n_k$. Then, for $m \in \mathcal{A}(S)$, we have
		\begin{enumerate}
			\item if $L(s+m) = L(s) + 1$ for all $s \in S$, then $m = n_1$, and
			\item if $\ell(s+m) = \ell(s) + 1$ for all $s \in S$, then $m = n_k$.
		\end{enumerate}
	\end{proposition}
	
	\begin{proof}
		Let $m \in \mathcal{A}(S)$ and suppose that $L(s+m) = L(s) + 1$ for all $s \in S$. We know that $m = n_i$ for some $i$. Suppose that $i \neq 1$. We first note that $L(n_1m) \geq m$. However, the sum of $m-1$ atoms is at least $(m-1)n_1$ and we see that $(m-1)n_1 = mn_1 - n_1 > -m + n_1m$. This means that $-m + n_1m$ must be factored as the sum of strictly fewer than $m - 1$ atoms. Thus, $L(-m+n_1m) + 1 < m \leq L(n_1m)$. This shows that if $L(s+m) = L(s) + 1$ for all $s \in S$, then $m = n_1$. The proof of the statement about $\ell(s)$ is similar. 
	\end{proof}
	
	Suppose we have $S \coloneqq \langle n_1, \dots, n_k \rangle$, a numerical semigroup minimally generated by $n_1, \dots, n_k \in \N$ with $n_1 < \cdots < n_k$. For $m = n_1$ or $m = n_k$, one relevant property an element $(c_{n_i})$ of $\mathsf{MinRepl}_m(S)$ can have is the property of being \textbf{left-zero}. We say that $(c_{n_i})$ is left-zero, if there exists some $j > 1$ such that $c_{n_i} = 0$ for all $i$ such that $2 \leq i < j$ and $c_{n_i} > 0$ for all $i \geq j$. Symmetrically, we say that $(c_{n_i})$ is \textbf{right-zero} if there exists some $j < k$ such that $c_{n_i} = 0$ for all $i$ such that $k-1 \geq i > j$ and $c_{n_i} > 0$ for all $i \leq j$.

	\begin{theorem}\label{Thm:NumericalPotentialExceptions}
		Let $S \coloneqq \langle n_1, \dots, n_k \rangle$ be a numerical semigroup minimally generated by $n_1, \dots, n_k \in \N$ with $n_1 < \cdots < n_k$. Set
		\[
			N_1(S) \coloneqq \{s \in M_{n_1}^{(1)}(S) \mid \exists\, (c_{n_i}) \in \mathsf{MinRepl}_{n_1}(S) \text{ such that } \varphi_S(c_{n_i}) = s, (c_{n_i}) \text{ is not left-zero}\}
		\]
		and
		\[
			N_2(S) \coloneqq \{s \in M_{n_k}^{(2)}(S) \mid \exists\, (c_{n_i}) \in \mathsf{MinRepl}_{n_k}(S) \text{ such that } \varphi_S(c_{n_i}) = s, (c_{n_i}) \text{ is not right-zero}\}
		\]
		Then,
		\begin{enumerate}
			\item we have $L(s+n_1) = L(s) + 1$ for all $s \in S$ if and only if $L(s) = L(s-n_1) + 1$ for all $s \in N_1(S)$, and
			\item we have $\ell(s+n_k) = \ell(s) + 1$ for all $s \in S$ if and only if $\ell(s) = \ell(s-n_k) + 1$ for all $s \in N_2(S)$.
		\end{enumerate}
	\end{theorem}
	\begin{proof}
		Corollary \ref{Cor:Reduction} implies almost all of this result except the left-zero and right-zero conditions. 
		
		Let $s \in M_{n_1}^{(1)}(S) \setminus N_1(S)$. This means that for all $(c_{n_i}) \in \mathsf{MinRepl}_{n_1}(S)$ such that $\varphi_S(c_{n_i}) = s$, we have that $(c_{n_i})$ is left-zero. Now let $(c_{n_i}) \in \mathsf{MinRepl}_{n_1}(S)$ and let $(b_{n_i}) \in \mathsf{Z}(s)$ such that $b_{n_1} > 0$. Because $(c_{n_i})$ is left-zero, there exists $j > 1$ such that $c_{n_i} > 0$ for all $i \geq j$ and $c_{n_i} = 0$ for all $i < j$. Let $(b_{n_i}) \in \mathsf{Z}(s)$ such that $b_{n_1} > 0$. We must have $b_i = 0$ for all $i \geq j$ since otherwise $(c_{n_i}) \notin \mathsf{MinRepl}_m(S)$. Now since
		\[
			s = \sum_{i=j}^{k} c_{n_i} n_i = \sum_{i=1}^{j-1} b_{n_i} n_i,
		\]
		we can calculate that
		\[
			\abs{(c_{n_i})} n_j \leq \sum_{i=j}^{k} c_{n_i} n_i = \sum_{i=1}^{j-1} b_{n_i} n_i < \abs{(b_{n_i})} n_j,
		\]
		which implies that $\abs{(c_{n_i})} < \abs{(b_{n_i})}$. This means any factorization of $s$ of the longest length must have at least one $n_1$. By Lemma \ref{Lem:Noms}, we have that $L(s) = L(s - n_1) + 1$. This shows that it suffices to check whether $L(s) = L(s-n_1) + 1$ for all $s \in N_1(S)$ in order to show that $L(s+n_1) = L(s) + 1$ for all $s \in S$. 	The proof of the second statement is similar. 
	\end{proof}

	\begin{remark}
		If we considered a numerical semigroup minimally generated by three elements $n_1, n_2, n_3 \in \N$ with $n_1 < n_2 < n_3$, Theorem \ref{Thm:NumericalPotentialExceptions} specializes to Theorem \ref{Thm:LiuYap}, which is Theorem 3.3 of \cite{LiuYap}, which says that there is only one potential exception to check when analyzing $L(s+n_1) = L(s) + 1$ or $\ell(s+n_3) = \ell(s) + 1$. In that case, elements of $\mathsf{MinRepl}_{m}(S)$, where $m \in \{n_1, n_3\}$, have two coordinates. According the previous theorem, we only have to consider the elements where, depending on whether $m$ is the smallest or largest generator, the first or the second coordinate is $0$. By minimality, there is only one potential exception to check. 
	\end{remark}
	
	We end this section by giving an example using a numerical semigroup. We order the coordinates of elements of $\mathsf{MinRepl}_m(S)$ in ascending order according to their indices.
	
	\begin{example}
		Let $S \coloneqq \langle 10, 12, 21, 38 \rangle$. We compute that
		\[
			\mathsf{MinRepl}_{10}(S) = \{(4,0,0), (0,2,0), (0,0,2), (1,0,1) \}.
		\]
		Evaluating these elements using $\varphi_S$ gives us $48, 42, 72, 50$. Note that the only element of length greater than $2$ is $(4,0,0)$. Since $48$ is not divisible by $42, 72$, or $50$, we have that $M_{10}^{(1)}(S) = N_1(S) = \{48\}$. We compute that $L(48) = 4$, realized only by $48 = 4\cdot 12$, but $L(48-10) + 1 = 2$. Thus, it is not the case that $L(s + 10) = L(s) + 1$ for all $s \in S$. 
		
		For the shortest factorization length, we calculate that
		\[
			\mathsf{MinRepl}_{38}(S) = \{(5,0,0), (0,4,0), (0,0,4), (4,3,0), (2,0,2), (0,3,2), (1,2,2) \}.
		\]
		The evaluations are $50, 48, 84, 76, 62, 78, 76$. We find that the only divisibility relations between distinct elements are $48| 84$, $48|78$, and $50|62|84$. Thus, $M_{38}^{(2)}(S) = \{50, 48, 76\}$. We know that $50$ only corresponds to $(5,0,0)$ in $\mathsf{MinRepl}_{38}(S)$, and $(5, 0, 0)$ is right-zero. The element $76$ corresponds to both $(4,3,0)$ and $(1,2,2)$ in $\mathsf{MinRepl}_{38}(S)$, both of which are right-zero, so $N_2(S) = \{48\}$. We find that $\ell(48) = 2$, realized by $48 = 10 + 38$, and $\ell(38) + 1 = 2$. This results in $\ell(48) = \ell(48-10) + 1$, which means we have that $\ell(s + 38) = \ell(s) + 1$ for all $s \in S$ according to Theorem \ref{Thm:NumericalPotentialExceptions}. 
	\end{example}

\section{Kunz preordered sets and Kunz polytopes}\label{Sect:Kunz}

\indent\indent Each numerical semigroup is associated to a Kunz poset. Numerical semigroups with the same Kunz poset have similar behaviors. Numerical semigroups with the same Kunz poset also correspond to integer points of the interior of the same face of a Kunz polytope. Since this way of grouping numerical semigroups together in families has been effective, we want to generalize the notions of Kunz posets and Kunz polytopes in the setting of commutative, cancellative semigroups. This will allow us to analyze the longest and shortest factorization lengths of elements of semigroups coming from the same family characterized by this generalized notion of Kunz posets. Not requiring the semigroups to be reduced will actually yield a preordered set instead of a poset.

\begin{definition}
Let $S$ be a commutative, cancellative semigroup. Take $m \in S$ such that $m \neq 0$. We define the \textbf{Apéry set} of $S$ with respect to $m$ as
\[
	\Ap(S;m) \coloneqq \{s \in S \mid s - m \notin S \}.
\]
Let $\pi: \mathsf{gp}(S) \to \mathsf{gp}(S)/\langle m \rangle$ denote the canonical projection. We then define \textbf{Kunz preordered set} of $S$ with respect to $m$ to be the preordered set
\[
	\mathsf{Kunz}(S;m) \coloneqq (\pi(\Ap(S;m)), \preceq),
\]
where for $s, t \in \Ap(S;m)$, we have $\pi(s) \preceq \pi(t)$ if and only if $t - s \in S$. Note that each element of $\pi(\Ap(S;m))$ has a unique preimage under $\pi$. If for some $s, t \in \Ap(S;m)$ we have $\pi(s) = \pi(t)$, then $s - t = cm$ for some $c \in \Z$. Without loss of generality, we can assume that $c \in \N$. Then $s = t + cm \in \Ap(S;m)$, meaning $c$ must be $0$ and therefore $s = t$. 
\end{definition}

\begin{remark}
	In general, $\preceq$ is a preorder on $\pi(\Ap(S;m))$. If $S$ is reduced, then $\preceq$ is a partial order. 
\end{remark}

For a numerical semigroup $S$ and a nonzero element $m \in S$, we have that $\pi(\Ap(S;m)) = \mathsf{gp}(S)/\langle m \rangle$, but for a general commutative, cancellative semigroup, $\pi(\Ap(S;m))$ might not be all of $\mathsf{gp}(S)/\langle m \rangle$. Notice also that for every $s$ in a numerical semigroup $S$, there is a maximal $p \in \N$ such that $s - pm \in S$. For a general commutative, cancellative semigroup $S$, we say that $s \in S$ is \textbf{infinitely divisible} by $m$ if $s - pm \in S$ for all $p \in \N$. If $m$ is not a  unit for $S$ and $L(s)$ is finite for all $s \in S$, then there does not exist an $s \in S$ such that $s$ is infinitely divisible by $m$. This yields $\pi(S) = \pi(\Ap(S;m))$, allowing the Apéry set to fully capture every residue class modulo $m$. This also forces $\pi(\Ap(S;m))$ to be a subsemigroup of $\mathsf{gp}(S)/\langle m \rangle$. 

\begin{lemma}
	Let $S$ be a commutative, cancellative semigroup and $m \in S$ be a nonzero element such that there does not exist $s \in S$ that is infinitely divisible by $m$. Then $\pi(\Ap(S;m))$ is a subsemigroup of $\mathsf{gp}(S)/\langle m \rangle$. 
\end{lemma}

\begin{proof}
	Let $\alpha, \beta \in \pi(\Ap(S;m))$. We know that there are unique $s, t \in \Ap(S;m)$ such that $\pi(s) = \alpha$ and $\pi(t) = \beta$. By assumption, there is some maximal $n \in \N$ such that $s + t - nm \in S$. Observe that $\pi(s+t-nm) = \alpha + \beta$. Furthermore, $s+t-nm \in \Ap(S;m)$. This means that $\alpha + \beta \in \pi(\Ap(S;m))$. Thus, $\pi(\Ap(S;m))$ is a subsemigroup of $\mathsf{gp}(S)/\langle m \rangle$. 
\end{proof}

Now suppose we have a preordered set $\mathcal{P}$ that is a Kunz preordered set of some commutative, cancellative semigroup with respect to some nonzero element $m$. Suppose we also know what the group of differences $G$ of this semigroup. We want to know which semigroups give rise to this Kunz preordered set. Take the canonical projection $\pi: G \to G/\langle m \rangle$ and choose $r_\alpha \in G$ such that $\pi(r_\alpha) = \alpha$ for each $\alpha \in G/\langle m \rangle$. Suppose that $\mathcal{P} = \mathsf{Kunz}(S;m)$ for some commutative, cancellative semigroup $S$. Take an element $\alpha$ in $\mathcal{P}$. This element has the form $\pi(s)$ for some $s \in \Ap(S;m)$. Since $\pi(s) = \alpha$, we have that $s = x_\alpha m + r_\alpha$ for some unique $x_\alpha \in \Z$. However, not every choice of $x_\alpha \in \Z$ for all $\alpha \in \mathcal{P}$ can recover the Apéry set of a semigroup because the elements of the form $x_\alpha m + r_\alpha$ have to be in the Apéry set. Let $\alpha, \beta \in \mathcal{P}$. Then $(x_\alpha m + r_\alpha) + (x_\beta m + r_\beta) = (x_\alpha + x_\beta + c_{\alpha, \beta})m + r_{\alpha+\beta} \in S$, where $d_{\alpha, \beta} \in \Z$ is the unique integer such that $r_\alpha + r_\beta = d_{\alpha, \beta}m + r_{\alpha + \beta}$. We also know that $x_{\alpha+\beta} \in \Z$ is such that $x_{\alpha+\beta}m + r_{\alpha+\beta} \in \Ap(S;m)$. This forces $x_\alpha + x_\beta + d_{\alpha, \beta} \geq x_{\alpha + \beta}$ since otherwise $x_{\alpha + \beta}m + r_{\alpha + \beta} \notin \Ap(S;m)$. This leads to the following generalization of a Kunz polytope. 

For the remainder of the section we let $(G, +)$ be an abelian group and $m \in G$ be a nonzero element. Also let $\pi: G \to G/\langle m \rangle$ be the canonical projection. Take $A$ to be some subsemigroup of $G/\langle m \rangle$. For each $\alpha \in A$, we fix $r_\alpha \in G$ such that $\pi(r_\alpha) = \alpha$. Additionally, for each $\alpha, \beta \in A$, we define $d_{\alpha, \beta} \in \Z$ to be the unique integer such that $r_\alpha + r_\beta = d_{\alpha, \beta}m + r_{\alpha + \beta}$ and more generally, for each $(c_\alpha) \in \N^A$, we define $d_{(c_\alpha)} \in \Z$ to be the unique integer such that
\[
	\sum_{\alpha \in A} c_\alpha r_\alpha = d_{(c_\alpha)}m + r_\beta, 
\]
where $\beta = \sum_{\alpha \in A} c_\alpha \alpha$. 

\begin{definition}
	We define
	\[
		P_{G, m, A, \{r_\alpha\}} \coloneqq \left\{ (x_\alpha) \in \prod_{\alpha \in A} \R \,\middle\vert\, x_\alpha + x_\beta + d_{\alpha, \beta} \geq x_{\alpha+\beta} \text{ for all } \alpha, \beta \in A   \right\},
	\] 
	to be the \textbf{Kunz polytope} with respect to $G$, $m$, $A$, and $\{r_\alpha\}$. We say that a Kunz polytope $P_{G, m, A, \{r_\alpha\}}$ is \textbf{admissible} if there is some integer point $(x_\alpha) \in P_{G, m, A, \{r_\alpha\}}$ such that $S \coloneqq \langle \{m\} \cup \{x_\alpha m + r_\alpha \mid \alpha \in A\} \rangle$ is a semigroup with the property that $G = \mathsf{gp}(S)$, $A = \pi(\Ap(S;m))$, and there does not exist $s \in S$ infinitely divisible by $m$. 
\end{definition}

Even though the inequalities that define the Kunz polytope only involve the sums of two elements of $A$, these inequalities imply inequalities that involve the sum of any finite number of elements of $A$. 

\begin{lemma}\label{Lem:IteratedInequality}
	For each $(c_\alpha) \in \N^A$, we have
	\[
	d_{(c_\alpha)} + \sum_{\alpha \in A} c_\alpha x_\alpha \geq x_\beta,
	\]
	where $\beta = \sum_{\alpha \in A} c_\alpha \alpha$, for each $(x_\alpha) \in P_{G, m, A, \{r_\alpha\}}$. Furthermore, $d_{(c_\alpha)} + d_{(c_\alpha')} + d_{\beta, \beta'} = d_{(c_\alpha + c_\alpha')} $ for each $(c_\alpha), (c_\alpha') \in \N^A$, where $\beta = \sum_{\alpha \in A} c_\alpha \alpha$ and $\beta' = \sum_{\alpha \in A} c_\alpha' \alpha$.
\end{lemma}

\begin{proof}
	We prove the second statement first. Let $(c_\alpha), (c_\alpha') \in \N^A$ and set $\beta \coloneqq \sum_{\alpha \in A} c_\alpha \alpha$ and $\beta' \coloneqq \sum_{\alpha \in A} c_\alpha' \alpha$. We know that
	\[
		\sum_{\alpha \in A} c_\alpha r_\alpha = d_{(c_\alpha)}m + r_\beta  \quad\text{and}\quad  \sum_{\alpha \in A} c_\alpha' r_\alpha = d_{(c_\alpha')}m + r_{\beta'}.
	\]
	Summing the two together yields
	\[
		\sum_{\alpha \in A} (c_\alpha + c_\alpha') r_\alpha = (d_{(c_\alpha)} + d_{(c_\alpha')})m + r_\beta + r_{\beta'}  = (d_{(c_\alpha)} + d_{(c_\alpha')} + d_{\beta, \beta'})m + r_{\beta + \beta'}.
	\]
	Thus, $d_{(c_\alpha + c_\alpha')} = d_{(c_\alpha)} + d_{(c_\alpha')} + d_{\beta, \beta'}$. This proves the second statement. 
	
	We now want to show the first statement. Take $(c_\alpha) \in \N^A$ and $(x_\alpha) \in P_{G, m, A, \{r_\alpha\}}$. If $\abs{(c_\alpha)} = 1$, then $d_{(c_\alpha)} = 0$ and the inequality holds. Suppose now that $\abs{(c_\alpha)} > 1$. Take $\gamma \in A$ such that $c_\gamma > 0$. Then by induction, we have
	\[
		d_{(c_\alpha) - e_\gamma} + (c_\gamma - 1)x_\gamma + \sum_{\alpha \in A \setminus \{\gamma\}} c_\alpha x_\alpha \geq x_{\beta - \gamma},
	\]
	where $\beta = \sum_{\alpha \in A} c_\alpha \alpha$. Since $(x_\alpha) \in P_{G, m, A, \{r_\alpha\}}$, we know that $x_\gamma + x_{\beta - \gamma} + d_{\gamma, \beta - \gamma} \geq x_\beta$ and therefore
	\[
		x_\beta \leq d_{(c_\alpha) - e_\gamma} + d_{\gamma, \beta - \gamma} + \sum_{\alpha \in A} c_\alpha x_\alpha =  d_{(c_\alpha) - e_\gamma} + d_{e_\gamma}+ d_{\gamma, \beta - \gamma} + \sum_{\alpha \in A} c_\alpha x_\alpha = d_{(c_\alpha)} + \sum_{\alpha \in A} c_\alpha x_\alpha 
	\]
	by the second statement of the lemma. 
\end{proof}

In the setting of numerical semigroups, we have $G = \Z$ and $A = \Z/m\Z$. We can also canonically use $r_{n+m\Z} = n$ for each $n \in \{0,1,2,\dots, m-1\}$. We keep the coordinate indexed by $0$ for convenience of calculations. This is different from the conventional Kunz polytope where the coordinate indexed by $0$ is omitted. Additionally, the addition of the coordinate indexed by $0$ gives extra conditions. In the setting of numerical semigroups, for each $n \in \{1, 2, \dots, m-1\}$, we have $d_{m e_{n+m\Z}} + m x_{n+m\Z} \geq x_0 = 0$ for each $(x_{\alpha}) \in P_{G, m, A, \{r_\alpha\}}$ by Lemma \ref{Lem:IteratedInequality}. Since $d_{m e_{n+m\Z}} = n < m$, we must have that $x_{n+m\Z} \geq 0$ for all $n$ for each integer point $(x_\alpha) \in P_{G, m, A, \{r_\alpha\}}$.

The following is a generalization of the correspondence between integer points of a Kunz polytope and numerical semigroups containing $m$, found in \cite[Theorem 2.6(a)]{Kunz3} and \cite{Kunz}. 

\begin{proposition}\label{Prop:PolytopeBijection}
	Suppose that $P_{G, m, A, \{r_\alpha\}}$ is admissible. Then there is a bijection 
	\[
		\rho: P_{G, m, A, \{r_\alpha\}} \cap \prod_{\alpha \in A} \Z \to \left\{ S \,\middle\vert\, \stackbox[l][c]{
			$S$ is a commutative, cancellative semigroup,\\
			$\mathsf{gp}(S) = G, m \in S, A = \pi(\Ap(S;m)),$ and\\
			there does not exist $s \in S$ infinitely divisible by $m$.
		}  \right\}.
	\]
	Furthermore, if $(x_\alpha)$ is an integer point of $P_{G, m, A, \{r_\alpha\}}$, then $\Ap(S;m) = \{ x_\alpha m + r_\alpha \mid \alpha \in A \}$. 
\end{proposition}

\begin{proof}
	Since $P_{G,m,A,\{r_\alpha\}}$ is admissible, there exists some integer point $(y_\alpha) \in P_{G,m,A,\{r_\alpha\}}$ such that $T \coloneqq \langle \{m\} \cup \{y_\alpha m + r_\alpha \mid \alpha \in A\} \rangle$ has the property that $\mathsf{gp}(T) = G$, $A = \pi(\Ap(T;m))$, and there does not exist an element $t \in T$ that is infinitely divisible by $m$. Let $(x_\alpha) \in P_{G, m, A, \{r_\alpha\}}$ be an integer point. We define $\rho(x_\alpha)$ to be \[S \coloneqq \langle \{m\} \cup \{x_\alpha m + r_\alpha \mid \alpha \in A\} \rangle.\] We need to ensure that $G = \mathsf{gp}(S)$, $\pi(\Ap(S;m)) = A$, and that there does not exist $s \in S$ infinitely divisible by $m$. Let $\alpha \in A$. We know that $y_\alpha m + r_\alpha \in \mathsf{gp}(S)$ from the generators of $S$. Since $T = \langle \{m\} \cup \{y_\alpha m + r_\alpha \mid \alpha \in A\} \rangle$ and $\mathsf{gp}(T) = G$, we know that $\mathsf{gp}(S) = G$ as well. Now we want to show that $\Ap(S;m) = \{ x_\alpha m + r_\alpha \mid \alpha \in A \}$. Let $\beta \in A$ and $n \in \Z$ such that $nm +r_\beta \in S$. Then using some $b \in \N$ and $(c_\alpha) \in \N^A$, we can write
	\[
		n m + r_\beta = bm + \sum_{\alpha \in A} c_\alpha (x_\alpha m + r_\alpha) = \left( b + d_{(c_\alpha)} + \sum_{\alpha \in A} c_\alpha x_\alpha \right) m + r_\beta.
	\]
	This shows that $n = b + d_{(c_\alpha)} + \sum_{\alpha \in A} c_\alpha x_\alpha$. By Lemma \ref{Lem:IteratedInequality}, we have that $n \geq b + x_\beta \geq x_\beta$. Therefore, $x_\beta m + r_\beta \in \Ap(S;m)$. On the other hand, let $s \in \Ap(S;m)$. Since $s$ is a sum of elements whose projections under $\pi$ are in $A$, we have $\pi(s) \in A$. This forces $s = x_{\pi(s)} m + r_{\pi(s)}$. Therefore, we have shown that \[\Ap(S;m) = \{ x_\alpha m + r_\alpha \mid \alpha \in A \}.\] Consequently, we have that $\pi(\Ap(S;m)) = A$.	Next, we let $s \in S$. Again, we have $\pi(s) \in A$. Therefore, we can write $s = n m + r_{\pi(s)}$ for some $n \in \N$. We have seen that $n \geq x_{\pi(s)}$, so it is not possible that $s - pm \in S$ for all $p \in \N$. We have shown that $\rho$ is well-defined. 
	 
	 To show that $\rho$ is surjective, we take a commutative, cancellative semigroup $S$ such that $G = \mathsf{gp}(S)$, $m \in S$, $\pi(\Ap(S;m)) = A$, and there does not exist $s \in S$ infinitely divisible by $m$. For each $\alpha \in A$, we have a unique $s \in \Ap(S;m)$ such that $\pi(s) = \alpha$, meaning there is a unique $x_\alpha \in \Z$ such that $s = x_\alpha m + r_\alpha$. We want to show that $(x_\alpha)$ is an integer point in $P_{G, m, A, \{r_\alpha\}}$. Let $\alpha, \beta \in A$. Then 
	 \[
	 	(x_\alpha m + r_\alpha) + (x_\beta m + r_\beta) = (x_\alpha + x_\beta + d_{\alpha, \beta})m + r_{\alpha + \beta}
	 \]
	 means that for $(x_\alpha m + r_\alpha) + (x_\beta m + r_\beta) \in S$ to be true, we must have $x_\alpha + x_\beta + d_{\alpha, \beta} \geq x_{\alpha + \beta}$. This shows that $(x_\alpha)$ is an integer point in $P_{G, m, A, \{r_\alpha\}}$. We want to show that $S$ is the image of the integer point $(x_\alpha) \in P_{G, m, A, \{r_\alpha\}}$ under $\rho$. We know that $\langle \{m\} \cup \{x_\alpha m + r_\alpha \mid \alpha \in A\} \rangle \subseteq S$. To show the reverse inclusion, we take $s \in S$. We know that there exists some $p \in \N$ such that $s - pm \in S$ but $s - (p+1)m \notin S$. Thus, $\pi(s-pm) \in A$, which means that $s-pm \in \{x_\alpha m + r_\alpha \mid \alpha \in A\}$ and thus $s \in \langle \{m\} \cup \{x_\alpha m + r_\alpha \mid \alpha \in A\} \rangle$. This shows that the map $\rho$ is surjective. 

	In order to show that $\rho$ is injective, we take $(x_\alpha), (x_\alpha') \in P_{G, m, A, \{r_\alpha\}}$ and suppose that $\rho(x_\alpha) = \rho(x_\alpha')$. For each $\alpha \in A$, we have seen that $x_\alpha m + r_\alpha$ and $x_\alpha' m + r_\alpha$ are in $\Ap(\mathsf{Kunz}(\rho(x_\alpha); m))$. Since both $x_\alpha m + r_\alpha$ and $x_\alpha' m + r_\alpha$ project to $\alpha$ under $\pi$, we must have $x_\alpha m + r_\alpha = x_\alpha' m + r_\alpha$, which implies that $x_\alpha = x_\alpha'$. This shows that the correspondence is injective. 
\end{proof}

To access the preorder that is defined on a Kunz preordered set, we need to consider the faces of the Kunz polytope. Just like the case with numerical semigroups, integer points of the interior of a face of the Kunz polytope correspond to semigroups with the same Kunz preordered set. 

\begin{definition}
	A \textbf{face} $F$ of the Kunz polytope $P_{G, m, A, \{r_\alpha\}}$ is a subset of the form \[\{(x_\alpha) \in P_{G, m, A, \{r_\alpha\}} \mid x_\alpha + x_\beta + c_{\alpha, \beta} = x_{\alpha + \beta} \text{ for all } (\alpha, \beta) \in C \},\] where $C$ is some subset of $A \times A$. We define the \textbf{interior} of a face $F$ to be all the points not contained in a face that is properly contained in $F$. 
\end{definition}

We use the correspondence between integer points and semigroups established in Proposition \ref{Prop:PolytopeBijection} to show that semigroups corresponding to the integer points in the interior of the same face of the Kunz polytope have the same Kunz preordered set. 

\begin{proposition}\label{Prop:FaceCorrespondence}
	Suppose that the Kunz polytope $P_{G, m, A, \{r_\alpha\}}$ is admissible. Two integer points $(x_\alpha), (x_\alpha')$ of $P_{G, m, A, \{r_\alpha\}}$ lie on the interior of the same face $F$ of $P_{G, m, A, \{r_\alpha\}}$, if and only if
	\[
		\mathsf{Kunz}(\rho(x_\alpha); m) = \mathsf{Kunz}(\rho(x_\alpha'); m)
	\]
	as preordered sets. Furthermore, for any face $F$ of $P_{G, m, A, \{r_\alpha\}}$ whose interior contains an integer point $(x_\alpha)$ and for $\alpha, \beta \in A$, we have $\alpha \preceq \beta$ in $\mathsf{Kunz}(\rho(x_\alpha); m)$ if and only if $y_\alpha + y_{\beta-\alpha} + d_{\alpha, \beta- \alpha} = y_\beta$ for all $(y_\alpha) \in F$. 
\end{proposition}

\begin{proof}
	Suppose that for $\alpha, \beta \in A$, we have $\alpha \preceq \beta$ in $\mathsf{Kunz}(\rho(x_\alpha); m)$. We know that $\Ap(\rho((x_\alpha));m) = \{ x_\alpha m + r_\alpha \mid \alpha \in A \}$ by Proposition \ref{Prop:PolytopeBijection}. Then $\alpha \preceq \beta$ implies that $(x_\beta m + r_\beta) - (x_\alpha m + r_\alpha) \in S$. Since $\pi((x_\beta m + r_\beta) - (x_\alpha m + r_\alpha) ) = \beta - \alpha$ and there is a maximal $p$ such that $(x_\beta m + r_\beta) - (x_\alpha m + r_\alpha)  -pm \in S$, we have that $\beta - \alpha \in A$. Due to the fact that $r_\alpha + r_{\beta-\alpha} = d_{\alpha, \beta-\alpha} m + r_\beta$, we get that $r_\beta - r_\alpha = -d_{\alpha, \beta-\alpha}m + r_{\beta-\alpha}$. Now we calculate that 
	\begin{align*}
			(x_\beta m + r_\beta) - (x_\alpha m + r_\alpha) &= (x_\beta - x_\alpha)m + (r_\beta - r_\alpha)\\
			& = (x_\beta - x_\alpha - d_{\alpha, \beta-\alpha})m + r_{\beta- \alpha}. 
	\end{align*}
	Since $(x_\alpha) \in P_{G, m, A, \{r_\alpha\}}$, we have that $x_\alpha + x_{\beta-\alpha} + d_{\alpha, \beta- \alpha} \geq x_\beta$. If this inequality is strict, then $x_\beta - x_\alpha - d_{\alpha, \beta-\alpha} < x_{\beta-\alpha}$, so $(x_\beta m + r_\beta) - (x_\alpha m + r_\alpha) \notin S$. Thus, $\alpha \preceq \beta$ implies that $x_\alpha + x_{\beta-\alpha} + d_{\alpha, \beta- \alpha} = x_\beta$. Therefore, $\mathsf{Kunz}(\rho(x_\alpha); m)$ and $\mathsf{Kunz}(\rho(x_\alpha'); m)$ have the same relations in their partial orders if and only if $(x_\alpha)$ and $(x_\alpha')$ lie on the interior of the same face of $P_{G,m,A,\{r_\alpha\}}$. 
	
	Let $F$ be a face of $P_{G, m, A, \{r_\alpha\}}$ whose interior contains an integer point $(x_\alpha)$. Also let $\alpha, \beta \in A$. We either have $y_\alpha + y_{\beta-\alpha} + d_{\alpha, \beta- \alpha} = y_\beta$ for all $(y_\alpha)$ in the interior of $F$ or $y_\alpha + y_{\beta-\alpha} + d_{\alpha, \beta- \alpha} > y_\beta$ for all $(y_\alpha)$ in the interior of $F$. We know that $\alpha \preceq \beta$ in $\mathsf{Kunz}(\rho(x_\alpha);m)$ if and only if $x_\alpha + x_{\beta-\alpha} + d_{\alpha, \beta- \alpha} = x_\beta$. This implies that $\alpha \preceq \beta$ in $\mathsf{Kunz}(\rho(x_\alpha); m)$ if and only if $y_\alpha + y_{\beta-\alpha} + d_{\alpha, \beta- \alpha} = y_\beta$ for all $(y_\alpha)$ in the interior of $F$, which is equivalent to saying that $y_\alpha + y_{\beta-\alpha} + d_{\alpha, \beta- \alpha} = y_\beta$ for all $(y_\alpha) \in F$.

\end{proof}

Now that we have seen that the integer points of the interior of a face of a Kunz polytope correspond to semigroups with the same Kunz preordered sets. We capture some of the behavior of the semigroups with the same Kunz preordered set by defining a semigroup operation on the Kunz preordered set. This definition is a generalization of the nilsemigroup defined in \cite[Theorem 3.3]{Kunz3}. Note for Kunz preordered sets of a general commutative, cancellative semigroup, we do not necessarily get an operation that makes every nonzero element nilpotent. 

\begin{definition}	
	Suppose that $\mathcal{P} \coloneqq \mathsf{Kunz}(S;m)$ is the Kunz preordered set for some commutative, cancellative semigroup $S$ and some nonzero $m \in S$ such that there does not exist an $s \in S$ infinitely divisible by $m$. We define a commutative semigroup operation $\oplus$ on the set $\mathcal{P}^\infty \coloneqq \mathcal{P} \cup \{\infty\}$. First, we have $\alpha \oplus \infty = \infty \oplus \alpha = \infty$ for all $\alpha \in \mathcal{P} \cup \{\infty\}$. Next, for each $\alpha, \beta \in \mathcal{P}$, we define
	\[
		\alpha \oplus \beta = \begin{cases}
			\alpha + \beta, & \text{if } \alpha \preceq \alpha + \beta \\
			\infty, &\text{otherwise.}
		\end{cases}
	\]
\end{definition}

We need to confirm that $\oplus$ is indeed a commutative semigroup operation on $\mathcal{P}^\infty$. We present the semigroup using the correspondence between integer points on the Kunz polytope and semigroups in Proposition \ref{Prop:PolytopeBijection} in order to utilize the translation of the partial order $\preceq$ into constraints on the Kunz polytope given by Proposition \ref{Prop:FaceCorrespondence}. 

\begin{lemma}
	Suppose that $P_{G, m, A, \{r_\alpha\}}$ is admissible. Let $(x_\alpha) \in P_{G, m, A, \{r_\alpha\}}$ be an integer point and set $\mathcal{P} \coloneqq \mathsf{Kunz}(\rho((x_\alpha));m)$. The operation $\oplus$ is a commutative semigroup operation on $\mathcal{P}^\infty$. Furthermore, the divisibility preordered set of $\mathcal{P}$ under $\oplus$ equals $(\mathcal{P}, \preceq)$ as preordered sets. 
\end{lemma}

\begin{proof}
	Let $\alpha, \beta, \gamma \in \mathcal{P}$. We want to show that $(\alpha \oplus \beta) \oplus \gamma = \alpha \oplus (\beta \oplus \gamma)$. 
	
	By Proposition \ref{Prop:FaceCorrespondence}, we have that $(\alpha \oplus \beta) \oplus \gamma = \alpha + \beta + \gamma$ if and only if $x_\alpha + x_\beta + d_{\alpha, \beta} = x_{\alpha+\beta}$ and $x_{\alpha + \beta} + x_\gamma + d_{\alpha+\beta, \gamma} = x_{\alpha+\beta+\gamma}$. This happens if and only if $x_{\alpha} + x_\beta + x_\gamma + d_{e_\alpha + e_\beta + e_\gamma} = x_{\alpha+ \beta + \gamma}$, which is invariant under permutations of $\{\alpha, \beta, \gamma\}$. Thus, we have $(\alpha \oplus \beta) \oplus \gamma = \alpha + \beta + \gamma$ if and only if  $\alpha \oplus (\beta \oplus \gamma) = \alpha + \beta + \gamma$. As a result, $(\alpha \oplus \beta) \oplus \gamma = \infty$ if and only if  $\alpha \oplus (\beta \oplus \gamma) = \infty$ as well.
	
	We have shown that $\oplus$ is associative. We also see $\alpha \preceq \alpha + \beta$ if and only if $x_{\alpha} + x_\beta + d_{\alpha, \beta} = x_{\alpha + \beta}$, which is invariant under permutations of $\{\alpha, \beta\}$. Thus, $\alpha \preceq \alpha + \beta$ is equivalent to $\beta \preceq \beta + \alpha$, meaning that $\oplus$ is commutative. 
	
	Lastly, we suppose that $\alpha, \beta \in \mathcal{P}$ such that $\alpha$ divides $\beta$ under $\oplus$. This means that there is some $\alpha' \in \mathcal{P}$ such that $\alpha \oplus \alpha' = \beta$, which means that $\alpha \preceq \beta$. Conversely, if $\alpha \preceq \beta$, then $(x_\beta m + r_\beta) - (x_\alpha m + r_\alpha) \in S$. Since the projection of this element under $\pi$ is $\beta - \alpha$, we have that $\beta - \alpha \in \mathcal{P}$. This shows that $\alpha \oplus (\beta - \alpha) = \beta$. Thus, the divisibility preordered set of $\mathcal{P}$ under $\oplus$ equals $(\mathcal{P}, \preceq)$ as preordered sets.
\end{proof}

Now that we have an operation on the Kunz preordered set that mimics the semigroup operation of a semigroup that has this particular Kunz preordered set, we can analyze the equations $L(s+m) = L(s) + 1$ and $\ell(s+m) = \ell(s) + 1$ using the operation on the Kunz preordered set. There is a face on the Kunz polytope such that every integer point in the interior of that face corresponds to a semigroup with the same Kunz preordered set according to Proposition \ref{Prop:FaceCorrespondence}.

\begin{lemma}\label{Lem:Translate}
	Suppose that $P_{G, m, A, \{r_\alpha\}}$ is admissible. Fix a face $F$ of $P_{G,m,A,\{r_\alpha\}}$. The interior of $F$ corresponds to some Kunz preordered set $\mathcal{P}$. Let $(x_\alpha)$ be an integer point in the interior of $F$. Set $S \coloneqq \rho(x_\alpha)$. We have bijections
	\begin{enumerate}
		\item $f: \mathcal{A}(\mathcal{P}^\infty_\textup{red}) \to \mathcal{A}(S_\textup{red}) \setminus \{[m]\}$ and
		\item $g: \min\limits_{\leq} \mathsf{Z}_{\mathcal{P}^\infty_\textup{red}}(\infty) \to \mathsf{MinRepl}_m(S)$. 
	\end{enumerate}
\end{lemma}

\begin{proof}
	We know that $0$ and $\infty$ are not atoms of $\mathcal{P}^\infty$, so we can write an element of $\mathcal{A}(\mathcal{P}^\infty_\textup{red})$ as $[\alpha]$ for some $\alpha \in \mathcal{P} \setminus \{0\}$. We define $f([\alpha]) = [x_\alpha m + r_\alpha]$. 
	
	First, we need to show that $f$ is well-defined. Let $[\alpha] \in \mathcal{A}(\mathcal{P}^\infty_\textup{red})$. To show that $[x_\alpha m + r_\alpha] \in \mathcal{A}(S_\textup{red}) \setminus \{[m]\}$, we prove that $x_\alpha m + r_\alpha$ is an atom of $S$ not associate to $m$. Suppose that $x_\alpha m + r_\alpha = s + t$ for some $s, t \in S$. Since $m$ does not infinitely divide $s$ or $t$ by Proposition \ref{Prop:PolytopeBijection}, we know that $\pi(s)$ and $\pi(t)$ are in $A$. We then write
	\[
		s = c_1 m + r_{\pi(s)} \quad\text{and}\quad t = c_2 m + r_{\pi(t)},
	\]
	where $c_1, c_2 \in \Z$. Since $s, t \in S$, we must have $c_1 \geq x_{\pi(s)}$ and $c_2 \geq x_{\pi(t)}$. Adding $s$ and $t$ together yields
	\[
		x_\alpha m + r_\alpha = (c_1 + c_2 + d_{\pi(s), \pi(t)}) m + r_\alpha.
	\]
	This leads to
	\[
		x_\alpha = c_1 + c_2 + d_{\pi(s), \pi(t)} \geq x_{\pi(s)} + x_{\pi(t)} + d_{\pi(s), \pi(t)} \geq x_\alpha,
	\]
	which forces $c_1 = x_{\pi(s)}$ and $c_2 = x_{\pi(t)}$. This shows that $s, t \in \Ap(S;m)$ by Proposition \ref{Prop:PolytopeBijection}. Since $(s + t) - s \in S$, we know that $\pi(s) \preceq \alpha$. Thus, $\pi(s) \oplus \pi(t) = \alpha$. Because $\alpha$ is an atom of $\mathcal{P}^\infty$, we can now see that $\pi(s)$ or $\pi(t)$ is a unit of $\mathcal{P}^\infty$. Without loss of generality, we assume that $\pi(s)$ is a unit of $\mathcal{P}^\infty$. This implies that there is some $\beta \in \mathcal{P}^\infty$ such that $\pi(s) \oplus \beta = 0$, which means $\pi(s) \preceq 0$. Now we have that $0 - s \in S$, showing that $s$ is a unit of $S$, which implies that $x_\alpha m + r_\alpha$ is an atom of $S$. Furthermore, $x_\alpha m + r_\alpha$ and $m$ cannot be associates, since otherwise, there exists some unit $u \in S$ such that $x_\alpha m + r_\alpha - m = u \in S$, contradicting the fact that $x_\alpha m + r_\alpha \in \Ap(S;m)$. We need to also show that the map $f$ does not depend on the choice of representative for $[\alpha]$. If $[\alpha] = [\alpha']$ for some $\alpha' \in \mathcal{P}^\infty$, then there exist some $\gamma, \gamma' \in \mathcal{P}^\infty$ such that $\alpha \oplus \gamma = \alpha'$ and $\alpha' \oplus \gamma' = \alpha$. This implies that $(x_{\alpha'} m + r_{\alpha'}) - (x_{\alpha} m + r_\alpha)$ and $(x_{\alpha} m + r_{\alpha}) - (x_{\alpha'} m + r_{\alpha'})$ are in $S$, which means $[x_\alpha m + r_\alpha] = [x_{\alpha'} m + r_{\alpha'}]$. 
	
	To show that $f$ is injective, we let $[\alpha], [\alpha'] \in \mathcal{P}^\infty_\textup{red}$ such that $[x_\alpha m + r_\alpha] = [x_{\alpha'} m + r_{\alpha'}]$. This implies that $(x_{\alpha'} m + r_{\alpha'}) - (x_{\alpha} m + r_\alpha) = x_{\alpha' - \alpha} m + r_{\alpha' - \alpha}$ and $(x_{\alpha} m + r_{\alpha}) - (x_{\alpha'} m + r_{\alpha'}) = x_{\alpha - \alpha'} m + r_{\alpha - \alpha'}$ are in $S$. Note now that we have $\alpha \preceq \alpha'$ and $\alpha' \preceq \alpha$. The fact that $\alpha \preceq \alpha'$ and $\alpha' \preceq \alpha$ implies both that $\alpha \oplus (\alpha' - \alpha) = \alpha'$ and $\alpha' \oplus (\alpha - \alpha') = \alpha$. To show that $\alpha' - \alpha$ is a unit in $\mathcal{P}^\infty$, we recognize that $x_{\alpha' - \alpha} m + r_{\alpha' - \alpha} + x_{\alpha - \alpha'} m + r_{\alpha - \alpha'} = 0$, so $(\alpha' - \alpha) \oplus (\alpha - \alpha') = 0$. Therefore, $\alpha \sim \alpha'$ in $\mathcal{P}^\infty$. 
	
	Next, we want to show that $f$ is surjective. Let $[s] \in \mathcal{A}(S_\textup{red}) \setminus \{[m]\}$. We claim that $f([\pi(s)]) = [s]$. Since $s$ is an atom of $s$ not associate to $m$, we have that $s \in \Ap(S; m)$. Therefore, $s = x_{\pi(s)} m + r_{\pi(s)}$. What is left to show is that $[\pi(s)]$ is an atom of $\mathcal{P}^\infty_\textup{red}$. This is achieved if we can show that $\pi(s)$ is an atom of $\mathcal{P}^\infty$. Let $\alpha, \beta \in \mathcal{P}^\infty$ such that $\alpha \oplus \beta = \pi(s)$. Then $(x_\alpha m + r_\alpha) + (x_\beta m + r_\beta) = s$ by Proposition \ref{Prop:FaceCorrespondence}, which means that $x_\alpha m + r_\alpha$ or $x_\beta m + r_\beta $ is a unit of $S$. This would imply that $\alpha$ or $\beta$ is a unit of $\mathcal{P}^\infty$, showing that $\pi(s)$ is an atom of $\mathcal{P}^\infty$ and therefore $[\pi(s)] \in \mathcal{A}(\mathcal{P}^\infty_\textup{red})$. 
	
	We now want to define a bijection $g: \min\limits_{\leq} \mathsf{Z}_{\mathcal{P}^\infty_\textup{red}}(\infty) \to \mathsf{MinRepl}_m(S)$. Let $(c_\alpha) \in \min\limits_{\leq} \mathsf{Z}_{\mathcal{P}^\infty_\textup{red}}(\infty)$. We define $g(c_\alpha) = (c_a)$, where for each $[a] \in \mathcal{A}(S_\textup{red}) \setminus \{[m]\}$, we define $c_a = c_{f^{-1}(a)}$. 
	
	To show that $g$ is well-defined, we take $(c_\alpha) \in \min\limits_{\leq} \mathsf{Z}_{\mathcal{P}^\infty_\textup{red}}(\infty)$. We want to show that $g(c_\alpha) = (c_a)$ is in $\mathsf{MinRepl}_m(S)$. By fixing a representative $\alpha \in A$ for each $[\alpha] \in \mathcal{P}^\infty_\textup{red} \setminus\{[\infty]\}$ and using Lemma \ref{Lem:IteratedInequality}, we have
	\[
		d_{(c_\alpha)} + \sum_{[\alpha] \in \mathcal{A}(\mathcal{P}^\infty_\textup{red})} c_\alpha x_\alpha \geq x_\beta,
	\]
	where $\beta = \varphi_A(c_\alpha)$. Due to Proposition \ref{Prop:FaceCorrespondence}, we cannot have equality since $\varphi_{\mathcal{P}^\infty}(c_\alpha) = \infty$. Now we have that 
	\[
		\sum_{[a] \in \mathcal{A}(S_\textup{red}) \setminus \{[m]\}} c_a a \sim \left(d_{(c_\alpha)} + \sum_{[\alpha] \in \mathcal{A}(\mathcal{P}^\infty_\textup{red})}c_\alpha x_\alpha\right)m + r_\beta,
	\]
	which is divisible by $m$ in $S$. Thus, $(c_a) \in \mathsf{Repl}_m(S)$. Now let $(c_a') \in \N^{\mathcal{A}(S_\textup{red})\setminus \{[m]\}} \setminus \{(c_a)\}$ such that $(c_a') \leq (c_a)$. Define $(c_\alpha') \in \N^{\mathcal{A}(\mathcal{P}^\infty_\textup{red})}$ to be such that $c_\alpha' = c_{f(\alpha)}'$ for each $[\alpha] \in \mathcal{A}(\mathcal{P}^\infty_\textup{red})$. This forces $(c_\alpha') \leq (c_\alpha)$ and $(c_\alpha') \neq (c_\alpha)$. By the minimality of $(c_\alpha)$ in $\mathsf{Z}_{\mathcal{P}^\infty_\textup{red}}(\infty)$, we must have that $\varphi_{\mathcal{P}^\infty}(c_\alpha') \neq \infty$. This shows that $\varphi_S(c_a') \in \Ap(S;m)$, which proves that $(c_a) \in \mathsf{MinRepl}_m(S)$. 
	
	Since $g$ is reindexing the coordinates of each element using a bijection $f$, we know that $g$ is injective. To show that $g$ is surjective, we take $(c_a) \in \mathsf{MinRepl}_m(S)$. Then we have $g(c_\alpha) = c_a$, where $c_\alpha = c_{f(\alpha)}$ for each $[\alpha] \in \mathcal{A}(\mathcal{P}^\infty_\textup{red})$. Using a method similar to the one we used to show that $g$ is well-defined, we can also show that $(c_\alpha) \in \min\limits_{\leq} \mathsf{Z}_{\mathcal{P}^\infty_\textup{red}}(\infty)$.
\end{proof}

This leads us to see that $\mathsf{MinRepl}_m(S) = \mathsf{MinRepl}_m(S')$ for semigroups $S$ and $S'$ corresponding to points in the interior of the same face of the Kunz polytope. Thus, we can use the same set $\mathsf{MinRepl}_m(S)$ to check for the validity of $L(s + m) = L(s) + 1$ for all $s \in S$ or for all $s \in S'$ according to Theorem \ref{Thm:PotentialExceptions}. 

\begin{lemma}
	Let $(x_\alpha)$ and $(x_\alpha')$ be integer points in the interior of the same face of $P_{G,m,A,\{r_\alpha\}}$. Then $\mathsf{MinRepl}_m(\rho(x_\alpha)) = \mathsf{MinRepl}_m(\rho(x_\alpha'))$.
\end{lemma}

Moving forward, we want to simplify the notation, so we want to only consider reduced semigroups. This is sufficient in terms of analyzing $L(s)$ and $\ell(s)$ since we can look at the factorization lengths in the reduced semigroup $S_\textup{red}$ instead of in $S$. What restricting to reduced semigroups means is that we restrict ourselves to certain faces of the Kunz polytope. The Kunz preordered set $\mathcal{P}$ and the operation $\oplus$ on $\mathcal{P}^\infty$ are both easier to handle in this case.

\begin{lemma}
	Let $(x_\alpha) \in P_{G, m, A, \{r_\alpha\}}$ be an integer point. Set $S \coloneqq \rho(x_\alpha)$. Then $S$ is reduced if and only if for all nonzero $\alpha \in A$ that is a unit in $A$, we have $x_\alpha + x_{-\alpha} + d_{\alpha, -\alpha} > x_0$. Additionally, in this case, $\mathcal{P} \coloneqq \mathsf{Kunz}(S;m)$ is a poset and $\mathcal{P}^\infty$ is reduced. 
\end{lemma}

\begin{proof}
	Suppose that $S$ is reduced. Then take a nonzero $\alpha \in A$ such that $\alpha$ is a unit in $A$. We have that $(x_\alpha m + r_\alpha) + (x_{-\alpha} m + r_{-\alpha}) = (x_\alpha + x_{-\alpha} + d_{\alpha, -\alpha})m + r_0$. We cannot have $x_\alpha + x_{-\alpha} + d_{\alpha, -\alpha} = x_0$ since otherwise $x_\alpha m + r_\alpha$ is a nonzero unit. Thus, for all nonzero $\alpha \in A$ that is a unit in $A$, we have $x_\alpha + x_{-\alpha} + d_{\alpha, -\alpha} > x_0$.
	
	Suppose that for all nonzero $\alpha \in A$ that are units in $A$, we have $x_\alpha + x_{-\alpha} + d_{\alpha, -\alpha} > x_0$. Take $u \in S$ to be a unit. Since neither $u$ nor $-u$ are infinitely divisible by $m$ by Proposition \ref{Prop:PolytopeBijection}, we know that both $\pi(u)$ and $\pi(-u) = -\pi(u)$ are in $A$. Because $u \in S$ is a unit, we have that $\pi(u) \preceq 0$ in $\mathcal{P}$, which implies that $x_{\pi(u)} + x_{-\pi(u)} + d_{\pi(u), -\pi(u)} = x_0$ by Proposition \ref{Prop:FaceCorrespondence}. By assumption, this forces $\pi(u) = 0$. Now we write $u = c_1m + r_0$ and $-u = c_2m + r_0$ for some $c_1, c_2 \in \Z$. Then we compute
	\[
		x_0m + r_0 = 0 = u + -u = c_1m + r_0 + c_2m + r_0 = (c_1 + c_2 + d_{0,0})m + r_0.
	\]
	This implies
	\[
		x_0 = c_1 + c_2 + d_{0,0} \geq x_0 + x_0 + d_{0,0} = x_0,
	\]
	which forces $u = -u = x_0m+r_0 = 0$. This means that $0$ is the only unit of $S$, so $S$ is reduced. 
	
	Now suppose that $S$ is reduced. We want to show that $\mathcal{P}$ is a poset. We already know that $\mathcal{P}$ is a preordered set, so we take $\alpha, \beta \in \mathcal{P}$ such that $\alpha \preceq \beta$ and $\beta \preceq \alpha$. We want to show that $\alpha = \beta$. By Proposition \ref{Prop:FaceCorrespondence}, we have that 
	\[
		x_\alpha + x_{\beta - \alpha} + d_{\alpha, \beta - \alpha} = x_\beta \quad\text{and}\quad x_\beta + x_{\alpha - \beta} + d_{\beta, \alpha - \beta} = x_\alpha.
	\]
	Combining the two equations yields
	\[
		x_{\beta - \alpha} + x_{\alpha - \beta} + d_{\alpha, \beta- \alpha} + d_{\beta, \alpha - \beta} = 0.
	\]
	We want to compute what $d_{\alpha, \beta- \alpha} + d_{\beta, \alpha - \beta} - d_{\beta-\alpha, \alpha - \beta}$ is. We know that
	\[
		r_\alpha + r_{\beta - \alpha} + r_\beta + r_{\alpha - \beta} - (r_{\beta-\alpha} + r_{\alpha - \beta}) = (d_{\alpha, \beta- \alpha} + d_{\beta, \alpha - \beta} - d_{\beta-\alpha, \alpha - \beta})m + r_\alpha + r_\beta - r_0.
	\]
	This implies that $(d_{\alpha, \beta- \alpha} + d_{\beta, \alpha - \beta} - d_{\beta-\alpha, \alpha - \beta})m = r_0$. Because $0 = x_0 m + r_0$, we can conclude that $d_{\alpha, \beta- \alpha} + d_{\beta, \alpha - \beta} - d_{\beta-\alpha, \alpha - \beta} = -x_0$. Thus, $x_{\beta - \alpha} + x_{\alpha - \beta} + d_{\alpha, \beta- \alpha} + d_{\beta, \alpha - \beta} = 0$ now becomes
	\[
		x_{\beta - \alpha} + x_{\alpha - \beta} + d_{\beta-\alpha, \alpha-\beta} - x_0 = 0.
	\]
	Since $S$ is reduced, this implies that $\beta - \alpha = 0$, which means $\alpha = \beta$. 
	
	Still supposing $S$ is reduced, we want to show that $\mathcal{P}^\infty$ is reduced. Let $\alpha \in \mathcal{P}^\infty$ be a unit. This means that $\alpha \oplus -\alpha = 0$, which implies that $\alpha \preceq 0$. Now we know that $0 - (x_\alpha m + r_\alpha) \in S$. Since $x_\alpha m + r_\alpha$ is also in $S$, we see that $x_\alpha m + r_\alpha$ is a unit of $S$. Thus, $x_\alpha m + r_\alpha = 0$ and therefore, $\alpha = 0$, showing that $\mathcal{P}^\infty$ is reduced. 
\end{proof}

Next, we wish to get an analogue of Corollary \ref{Cor:Reduction} that does not depend on a particular semigroup. To achieve this, we transfer divisibility of the elements of a reduced semigroup $S$ that are evaluations of elements of $\mathsf{MinRepl}_m(S)$ under $\varphi_S$ to a preorder $\sqsubseteq_S$ on $\min\limits_{\leq} \mathsf{Z}_{\mathcal{P}^\infty_\textup{red}}(\infty)$ derived from the Kunz poset $\mathcal{P} = \mathsf{Kunz}(S;m)$.

\begin{definition}
	Suppose that $P_{G, m, A, \{r_\alpha\}}$ is admissible. Let $(x_\alpha)$ be a point in the interior of $F$ such that $S \coloneqq \rho(x_\alpha)$ is reduced. Set $\mathcal{P} \coloneqq \mathsf{Kunz}(S;m)$. Take $(c_\alpha), (c_\alpha') \in \min\limits_{\leq} \mathsf{Z}_{\mathcal{P}^\infty}(\infty)$. Then we define the preorder $\sqsubseteq_S$ by
	\[(c_\alpha) \sqsubseteq_S (c_\alpha') \iff  - x_{\beta' - \beta} + \sum_{\alpha \in  \mathcal{A}(\mathcal{P}^\infty)} (c_\alpha' - c_\alpha) x_\alpha \geq b_{(c_\alpha), (c_\alpha')}
	\]
	where $\beta = \sum_{\alpha \in  \mathcal{A}(\mathcal{P}^\infty)} c_\alpha \alpha, \beta' = \sum_{\alpha \in \mathcal{A}(\mathcal{P}^\infty)} c_\alpha' \alpha$, and $b_{(c_\alpha), (c_\alpha')} \in \Z$ is the unique integer such that $b_{(c_\alpha), (c_\alpha')} m = r_{\beta' - \beta} + \sum_{\alpha \in \mathcal{A}(\mathcal{P}^\infty)} (c_\alpha - c_\alpha') r_\alpha$.
\end{definition}

We can show that $\sqsubseteq_S$ is indeed a preorder on $\min\limits_{\leq} \mathsf{Z}_{\mathcal{P}^\infty_\textup{red}}(\infty)$.

\begin{lemma}\label{Lem:Sq}
	Suppose that $P_{G, m, A, \{r_\alpha\}}$ is admissible. Let $(x_\alpha)$ be a point in the interior of $F$ such that $S \coloneqq \rho(x_\alpha)$ is reduced. Set $\mathcal{P} \coloneqq \mathsf{Kunz}(S;m)$. The relation $\sqsubseteq_S$ defines a preorder on $\min\limits_{\leq} \mathsf{Z}_{\mathcal{P}^\infty}(\infty)$. Plus, $(c_\alpha) \sqsubseteq_S (c_\alpha')$ if and only if \[ \sum_{\alpha \in \mathcal{A}(\mathcal{P}^\infty)}c_\alpha (x_\alpha m + r_\alpha) \text{ divides } \sum_{\alpha \in \mathcal{A}(\mathcal{P}^\infty)}c_{\alpha}' (x_\alpha m + r_\alpha) \] in $S$.
\end{lemma}

\begin{proof}
	If we can show the last statement, then we show that $\sqsubseteq_S$ defines a preorder on $\min\limits_{\leq} \mathsf{Z}_{\mathcal{P}^\infty}(\infty)$ since divisibility is a preorder on $S$.  
	
	Let $(c_\alpha), (c_\alpha') \in \min\limits_{\leq}\mathsf{Z}_{\mathcal{P}^\infty}(\infty)$. Set $\beta \coloneqq \sum_{\alpha \in  \mathcal{A}(\mathcal{P}^\infty)} c_\alpha \alpha$ and $\beta' \coloneqq \sum_{\alpha \in \mathcal{A}(\mathcal{P}^\infty)} c_\alpha' \alpha$. We calculate that
	\begin{align*}
		&\sum_{\alpha \in \mathcal{A}(\mathcal{P}^\infty)}c_{\alpha}' (x_\alpha m + r_\alpha) - \sum_{\alpha \in \mathcal{A}(\mathcal{P}^\infty)}c_\alpha (x_\alpha m + r_\alpha) \\&= \sum_{\alpha \in \mathcal{A}(\mathcal{P}^\infty)}(c_\alpha' - c_\alpha)x_\alpha m + \sum_{\alpha \in \mathcal{A}(\mathcal{P}^\infty)}(c_\alpha' - c_\alpha) r_\alpha\\
		& = \left(d_{(c_\alpha')} - d_{(c_\alpha)} + \sum_{\alpha \in \mathcal{A}(\mathcal{P}^\infty)}(c_\alpha' - c_\alpha)x_\alpha \right) m + r_{\beta'} - r_{\beta}\\
		& = \left(-d_{\beta, \beta' - \beta} + d_{(c_\alpha')} - d_{(c_\alpha)} + \sum_{\alpha \in \mathcal{A}(\mathcal{P}^\infty)}(c_\alpha' - c_\alpha)x_\alpha \right) m + r_{\beta' - \beta},
	\end{align*}
	which is in $S$ if and only if
	\[
		-d_{\beta, \beta' - \beta} + d_{(c_\alpha')} - d_{(c_\alpha)} + \sum_{\alpha \in \mathcal{A}(\mathcal{P}^\infty)}(c_\alpha' - c_\alpha)x_\alpha \geq x_{\beta' - \beta}.
	\]
	Rearranging provides us with
	\[
		-x_{\beta' - \beta} + \sum_{\alpha \in \mathcal{A}(\mathcal{P}^\infty)}(c_\alpha' - c_\alpha)x_\alpha \geq d_{\beta, \beta' - \beta} - d_{(c_\alpha')} + d_{(c_\alpha)}.
	\]
	We want to know what the right hand side is, so we compute that
	\begin{align*}
		(d_{\beta, \beta' - \beta} - d_{(c_\alpha')} + d_{(c_\alpha)})m + r_{\beta'} -r_{\beta'} + r_\beta &= r_\beta + r_{\beta' - \beta} - \sum_{\alpha \in \mathcal{A}(\mathcal{P}^\infty)}c_\alpha' r_\alpha +\sum_{\alpha \in \mathcal{A}(\mathcal{P}^\infty)}c_\alpha r_\alpha\\
		(d_{\beta, \beta' - \beta} - d_{(c_\alpha')} + d_{(c_\alpha)})m & =  r_{\beta' - \beta} + \sum_{\alpha \in \mathcal{A}(\mathcal{P}^\infty)}(c_\alpha - c_\alpha')\alpha.
	\end{align*}
	Therefore, $d_{\beta, \beta' - \beta} - d_{(c_\alpha')} + d_{(c_\alpha)} = b_{(c_\alpha), (c_\alpha')}$. This shows that $(c_\alpha) \sqsubseteq_S (c_\alpha')$ if and only if $\sum_{\alpha \in \mathcal{A}(\mathcal{P}^\infty)}c_\alpha (x_\alpha m + r_\alpha)$ divides $\sum_{\alpha \in \mathcal{A}(\mathcal{P}^\infty)}c_{\alpha}' (x_\alpha m + r_\alpha)$. 

\end{proof}

Reduced semigroups $S$ and $S'$ with the same Kunz poset $\mathcal{P}$ can induce different preorders on $\sqsubseteq_S$ and $\sqsubseteq_{S'}$ on $\min\limits_{\leq} \mathsf{Z}_{\mathcal{P}^\infty}(\infty)$. We can choose to divide up the interior of a face $F$ of $P_{G, m, A, \{r_\alpha\}}$ so that for each $(x_\alpha)$ and $(x_\alpha')$ in the same region, we have $\sqsubseteq_{\rho(x_\alpha)} = \sqsubseteq_{\rho(x_\alpha')}$. However, for the purposes of Corollary \ref{Cor:Reduction}, since $\sqsubseteq_S$ and $\sqsubseteq_{S'}$ translate divisibility in $S$ and $S'$, we only need to group together the integer points in the interior of $F$ that correspond to semigroups that define preorders on $\min\limits_{\leq} \mathsf{Z}_{\mathcal{P}^\infty}(\infty)$ with the same pseudominimal elements. The following lemma falls out of how the preorders on $\min\limits_{\leq} \mathsf{Z}_{\mathcal{P}^\infty}(\infty)$ are defined. 

\begin{lemma}\label{Lem:Cominimal}
	Suppose that $P_{G, m, A, \{r_\alpha\}}$ is admissible. Fix a face $F$ of $P_{G,m,A,\{r_\alpha\}}$. Let $(x_\alpha)$ be a point in the interior of $F$. Set $S \coloneqq \rho(x_\alpha)$ and $\mathcal{P} \coloneqq \mathsf{Kunz}(S;m)$. Let $(x_\alpha')$ be another integer point in the interior of $F$ and set $S' \coloneqq \rho(x_\alpha')$. We have that $\pmin\limits_{\sqsubseteq_S}\min\limits_{\leq} \mathsf{Z}_{\mathcal{P}^\infty}(\infty) = \pmin\limits_{\sqsubseteq_{S'}}\min\limits_{\leq} \mathsf{Z}_{\mathcal{P}^\infty}(\infty)$ if and only if for each $(c_\alpha) \in \pmin\limits_{\sqsubseteq_S}\min\limits_{\leq} \mathsf{Z}_{\mathcal{P}^\infty}(\infty)$, we have
	\begin{align*}
		&-x_{\sum(c_\alpha - c_\alpha')\alpha}' + \sum_{\alpha \in \mathcal{A}(\mathcal{P}^\infty)} (c_\alpha' - c_\alpha)x_\alpha' \geq b_{(c_\alpha'), (c_\alpha)} \\\implies& -x_{\sum(c_\alpha' - c_\alpha)\alpha}' + \sum_{\alpha \in \mathcal{A}(\mathcal{P}^\infty)} (c_\alpha - c_\alpha')x_\alpha' \geq b_{(c_\alpha), (c_\alpha')}
	\end{align*}
	for each $(c_\alpha') \in \min\limits_{\leq} \mathsf{Z}_{\mathcal{P}^\infty}(\infty) \setminus \{(c_\alpha)\}$; and for each $(c_\alpha) \in \min\limits_{\leq} \mathsf{Z}_{\mathcal{P}^\infty}(\infty) \setminus \min\limits_{\sqsubseteq_S}\min\limits_{\leq} \mathsf{Z}_{\mathcal{P}^\infty}(\infty)$, there exists $(c_\alpha') \in \min\limits_{\leq} \mathsf{Z}_{\mathcal{P}^\infty}(\infty)$ such that
	\begin{align*}
		&-x_{\sum(c_\alpha - c_\alpha')\alpha}' + \sum_{\alpha \in \mathcal{A}(\mathcal{P}^\infty)} (c_\alpha' - c_\alpha)x_\alpha' \geq b_{(c_\alpha'), (c_\alpha)} \\
		\text{and } & -x_{\sum(c_\alpha' - c_\alpha)\alpha}' + \sum_{\alpha \in \mathcal{A}(\mathcal{P}^\infty)} (c_\alpha - c_\alpha')x_\alpha' < b_{(c_\alpha), (c_\alpha')}.
	\end{align*}
\end{lemma}

\begin{definition}
	Let $(x_\alpha), (x_\alpha')$ be two integer points in the interior of the same face of $P_{G, m, A, \{r_\alpha\}}$. We say that $(x_\alpha), (x_\alpha')$ are \textbf{cominimal} if $\pmin\limits_{\sqsubseteq_{\rho(x_\alpha)}}\min\limits_{\leq} \mathsf{Z}_{\mathcal{P}^\infty}(\infty) = \pmin\limits_{\sqsubseteq_{\rho(x_\alpha')}}\min\limits_{\leq} \mathsf{Z}_{\mathcal{P}^\infty}(\infty)$
\end{definition}

Another component of Corollary \ref{Cor:Reduction} we need to worry about is that the element $m$ needs to be an atom and $L(s)$ is finite for all elements $s$ of the semigroup. We can translate the condition that $m$ is an atom to a condition for the corresponding point on the Kunz polytope provided that $L(s)$ is finite for all $s$. 

\begin{lemma}
	Suppose that $P_{G, m, A, \{r_\alpha\}}$ is admissible. Let $(x_\alpha)$ be an integer point in $P_{G, m, A, \{r_\alpha\}}$. Set $S \coloneqq \rho(x_\alpha)$. Suppose that $S$ is reduced and $L(s)$ is finite for all $s \in S$. Then $m$ is an atom of $S$ if and only if 
	\[
		-x_0 + \sum_{\alpha \in A} c_\alpha x_\alpha \neq 1 - d_{(c_\alpha)}
	\]
	for all $(c_\alpha) \in \N^A$ such that $\sum_{\alpha \in A} c_\alpha \alpha = 0$. 
\end{lemma}

\begin{proof}
	Suppose $m$ is an atom of $S$. Let $(c_\alpha) \in \N^A$ such that $\sum_{\alpha \in A} c_\alpha \alpha = 0$. If $\abs{(c_a)} < 2$, then $-x_0 + d_{(c_\alpha)} + \sum_{\alpha \in A} c_\alpha x_\alpha = 0 \neq 1$. Suppose now that $\abs{(c_a)} \geq 2$. Then using the fact that $0 = x_0m + r_0$, we get that 
	\begin{align*}
		\sum_{\alpha \in A} c_\alpha(x_\alpha m + r_\alpha) & = \sum_{\alpha \in A} c_\alpha x_\alpha m  + \sum_{\alpha \in A} c_\alpha r_\alpha \\& = \left( d_{(c_\alpha)} + \sum_{\alpha \in A} c_\alpha x_\alpha  \right)m + r_0 \\&= \left(-x_0 + d_{(c_\alpha)} + \sum_{\alpha \in A} c_\alpha x_\alpha  \right)m,
	\end{align*}
	which cannot be equal to $m$ since $m$ is an atom and $\abs{(c_\alpha)} \geq 2$. Thus, $-x_0 + \sum_{\alpha \in A} c_\alpha x_\alpha \neq 1 - d_{(c_\alpha)}$. 
	
	Now suppose that $-x_0 + \sum_{\alpha \in A} c_\alpha x_\alpha \neq 1 - d_{(c_\alpha)}$ for all $(c_\alpha) \in \N^A$ such that $\sum_{\alpha \in A} c_\alpha \alpha = 0$. From Proposition \ref{Prop:PolytopeBijection}, we have that the atoms of $S$ are of the form $x_\alpha m + r_\alpha$ for some $\alpha \in A$ and possibly $m$. Suppose that $m$ is not an atom of $S$. We can write 
	\[
		m = \sum_{\alpha \in A} c_\alpha (x_\alpha m + r_\alpha)
	\]
	for some $(c_\alpha) \in \N^A$ where $c_\alpha = 0$ if $x_\alpha m + r_\alpha$ is not an atom of $S$. Note that $\sum_{\alpha \in A} c_\alpha \alpha = 0$ and we have $-x_0 + \sum_{\alpha \in A} c_\alpha x_\alpha = 1 - d_{(c_\alpha)}$, a contradiction. Thus, $m$ must be an atom of $S$. 
\end{proof}

Now we have all the ingredients to translate Corollary \ref{Cor:Reduction} into the language of Kunz posets and Kunz polytopes. 

\begin{theorem}\label{Thm:Main}
	Suppose that $P_{G, m, A, \{r_\alpha\}}$ is admissible. Fix a face $F$ of $P_{G,m,A,\{r_\alpha\}}$. The interior of $F$ corresponds to some Kunz preordered set $\mathcal{P}$. Let $(x_\alpha)$ be an integer point in the interior of $F$ such that $S \coloneqq \rho(x_\alpha)$ is reduced, $m$ is an atom of $S$, and $L(s)$ is finite for all $s \in S$. Then
	\begin{enumerate}
		\item $L(s+m) = L(s) + 1$ for all $s \in S$ if and only if 
		\[
			- x_{\sum c_\alpha \alpha} + \sum_{\alpha \in \mathcal{A}(\mathcal{P}^\infty)} c_\alpha x_\alpha \geq \abs{(c_\alpha)} - d_{(c_\alpha)} - L_{\mathcal{P}^\infty}\left(\sum_{\alpha \in \mathcal{A}(\mathcal{P}^\infty)} c_\alpha \alpha \right) 
		\]
		for all $(c_\alpha) \in \pmin\limits_{\sqsubseteq_S}  \min\limits_{\leq} \mathsf{Z}_{\mathcal{P}^\infty}(\infty)$ such that $\abs{(c_\alpha)} > 2$, and
		\item $\ell(s+m) = \ell(s) + 1$ for all $s \in S$ if and only if 
		\[
			- x_{\sum c_\alpha \alpha} + \sum_{\alpha \in \mathcal{A}(\mathcal{P}^\infty)} c_\alpha x_\alpha \leq \abs{(c_\alpha)} - d_{(c_\alpha)} - \ell_{\mathcal{P}^\infty}\left(\sum_{\alpha \in \mathcal{A}(\mathcal{P}^\infty)} c_\alpha \alpha \right) 
		\]
		for all $(c_\alpha) \in \pmin\limits_{\sqsubseteq_S}  \min\limits_{\leq} \mathsf{Z}_{\mathcal{P}^\infty}(\infty)$. 
	\end{enumerate}
\end{theorem}

\begin{proof}
	Fix an integer point $(x_\alpha)$ in the interior of $F$. We make use of Corollary \ref{Cor:Reduction}. By Lemma \ref{Lem:Sq} and the bijection $g: \min\limits_{\leq}\mathsf{Z}_{\mathcal{P}^\infty} \to \mathsf{MinRepl}_m(S)$ from Lemma \ref{Lem:Translate}, we have that $L(s + m) = L(s) + 1$ for all $s \in S$ if and only if $L(\varphi_S(g(c_\alpha))) = L(\varphi_S(g(c_\alpha)) - m) + 1$ for all $(c_\alpha) \in \pmin\limits_{\sqsubseteq_S}\min\limits_{\leq}\mathsf{Z}_{\mathcal{P}^\infty}$ such that $\abs{(c_\alpha)} > 2$. 
	
	Take $(c_\alpha) \in \pmin\limits_{\sqsubseteq_S}\min\limits_{\leq}\mathsf{Z}_{\mathcal{P}^\infty}$ such that $\abs{(c_\alpha)} > 2$. Set $s \coloneqq \varphi_S(g(c_\alpha))$. Now we calculate $L(-m + s) + 1$. Let $(b_a) \in \mathsf{Z}(s)$ such that $b_m > 0$ and $\abs{(b_a)}$ is maximal among all sequences with the property that the coordinate indexed by $m$ is strictly positive. Note that $L(-m+s) + 1 = \abs{(b_a)}$. We know that $(b_a) - b_me_m \in \mathsf{Z}(-b_m m + s)$ so $L(-b_m m + s) \geq \abs{(b_a)} - b_m$. On the other hand, let $(b_a') \in \mathsf{Z}(-b_m m + s)$ such that $\abs{(b_a')} = L(-b_m m + s)$. Then $b_m e_m + (b_a') \in \mathsf{Z}(s)$, which is a factorization of $s$ with at least one $m$, so we also get that $\abs{(b_a)} \geq b_m + L(-b_m m + s)$. Therefore, we have that 
	\[
		L(-m + s) + 1 = b_m + L(-b_m m + s). 
	\]
	We now claim that $-b_m m + s \in \Ap(S;m)$. Suppose not. Then $-b_m m + s = \varphi_S(g(b_\alpha''))$ for some $(b_\alpha'') \in \mathsf{Z}_{\mathcal{P}^\infty}(\infty)$. Take $(b_\alpha''') \in \min\limits_{\leq} \mathsf{Z}_{\mathcal{P}^\infty}(\infty)$ such that $(b_\alpha''') \leq (b_\alpha'')$. Now we compute that
	\[
		s - \varphi_S(g(b_\alpha''')) = b_m m + \varphi_S(g(b_\alpha'') - g(b_\alpha''')) \in S.
	\]
	This implies that $(b_\alpha''') \sqsubseteq_S (c_\alpha)$ by Lemma \ref{Lem:Sq}. Since $(c_\alpha)$ is a pseudominimum, we also get that $(c_\alpha) \sqsubseteq_S (b_\alpha''')$. Therefore, $s = \varphi_S(g(b_\alpha'''))$. Since $b_m \neq 0$, we get that $m$ is a unit and $0$ would be infinitely divisible by $m$, a contradiction. Therefore, $-b_m m + s$ is an element of the Apéry set. Since $-b_m m + s \in \Ap(S;m)$ and $\pi(-b_m m + s) = \sum_{\alpha \in \mathcal{A}(\mathcal{P}^\infty)} c_\alpha \alpha$, we know that \[L_S(-b_m m + s) = L_{\mathcal{P}^\infty}\left(\sum_{\alpha \in \mathcal{A}(\mathcal{P}^\infty)} c_\alpha \alpha \right).\] As for calculating $b_m$, we compute
	\begin{align*}
		b_m m = s - (-b_m m + s) &= \sum_{\alpha \in \mathcal{A}(\mathcal{P}^\infty)} c_\alpha(x_\alpha m + r_\alpha) - (x_{\sum c_\alpha \alpha} m + r_{\sum c_\alpha \alpha})\\
		& = \left( d_{(c_\alpha)} - x_{\sum c_\alpha \alpha} + \sum_{\alpha \in \mathcal{A}(\mathcal{P}^\infty)} c_\alpha x_\alpha  \right)m.
	\end{align*}
	Therefore, 
	\[
			L(-m + s) + 1 = L_{\mathcal{P}^\infty}\left(\sum_{\alpha \in \mathcal{A}(\mathcal{P}^\infty)} c_\alpha \alpha \right)+ d_{(c_\alpha)} - x_{\sum c_\alpha \alpha} + \sum_{\alpha \in \mathcal{A}(\mathcal{P}^\infty)} c_\alpha x_\alpha. 
	\]
	
	Now we are ready to so the converse direction of the first statement. We assume that for all $(c_\alpha) \in \pmin\limits_{\sqsubseteq_S}  \min\limits_{\leq} \mathsf{Z}_{\mathcal{P}^\infty}(\infty)$ such that $\abs{(c_\alpha)} > 2$, we have that \[- x_{\sum c_\alpha \alpha} + \sum_{\alpha \in \mathcal{A}(\mathcal{P}^\infty)} c_\alpha x_\alpha \geq \abs{(c_\alpha)} - d_{(c_\alpha)} - L_{\mathcal{P}^\infty}\left(\sum_{\alpha \in \mathcal{A}(\mathcal{P}^\infty)} c_\alpha \alpha \right). \]
	Rearranging the inequality yields $\abs{(c_a)} \leq L(-m + s) + 1$, $s = \varphi(g(c_\alpha))$. Suppose that $(b_a) \in \mathsf{Z}(s)$ such that $\abs{(b_a)} = L(s)$. If $b_m > 0$, then $L(s) = L(-m + s) + 1$ by Lemma \ref{Lem:Noms}. If $b_m = 0$, then we can argue that $g^{-1}(\widehat{(b_a)})$, where $\widehat{(b_a)}$ is $(b_a)$ with the coordinate indexed by $m$ omitted, is in $\pmin\limits_{\sqsubseteq_S}  \min\limits_{\leq} \mathsf{Z}_{\mathcal{P}^\infty}(\infty)$. We know that $\abs{(b_a)} > 2$ so now we get $L(s) = \abs{(b_a)} \leq L(-m + s) + 1$, which implies $L(s) = L(-m + s) + 1$. 
	
	Next, we now suppose that there exists some $(c_\alpha) \in \pmin\limits_{\sqsubseteq_S}  \min\limits_{\leq} \mathsf{Z}_{\mathcal{P}^\infty}(\infty)$ with $\abs{(c_\alpha)} > 2$ such that $\abs{(c_\alpha)} > L(-m + s) + 1$, where $s = \varphi(g(c_\alpha))$. Then $L(s) = \abs{(c_\alpha)} > L(-m + s) + 1$, as desired.

	The proof for the shortest length formula is similar.
\end{proof}

	Note that for cominimal points, the same inequalities are used to detect whether $L(s+m) = L(s) + 1$ for all $s$ and similarly for for $\ell(s+m) = \ell(s) + 1$. We give an example of this theorem to test the validity of the formula $L(s + m) = L(s) + 1$ for all semigroups in a family. 

\begin{example}
	Let $S = \langle 5,6,8 \rangle$. We can compute that $\mathsf{MinRepl}_5(S) = \{(3,0), (0,2), (2,1)\}$. The evaluations of these elements under $\varphi_S$ are $18, 16$, and $20$. There are no nontrivial divisibility relations between these elements. We therefore have $\pmin\limits_{\sqsubseteq_S} \min\limits_{\leq}  \mathsf{Z}_{\mathcal{P}^\infty}(\infty) = \{(3,0), (0,2), (2,1)\}$, where $\mathcal{P} = \mathsf{Kunz}(S;5)$. 
	
	We compute that $\Ap(S;5) = \{0, 6, 12, 8, 14\}$ and thus $S$ corresponds to the point $(0, 1, 2, 1, 2)$ on $P_{\Z, 5, \Z/5\Z, \{0,1,2,3,4\}}$. We want to see which points are cominimal to $(0, 1, 2, 1, 2)$ on the interior of the same face $F$ of the Kunz polytope. We can compute that $(\mathcal{P}, \preceq)$ has the nontrivial relations $0 \preceq 1 \preceq 2$, $0 \preceq 3 \preceq 4$ and $1 \preceq 4$. 
	
	Suppose that $(x_0, x_1, x_2, x_3, x_4)$ is an integer point in the interior of $F$ that is cominimal with $(0,1,2,1,2)$. If $(3,0) \sqsubseteq_{\rho(x_\alpha)} (0,2)$, then $(0,2) \sqsubseteq_{\rho(x_\alpha)} (3,0)$, we have that $-x_3 + 2x_3 - 3x_1 \geq 0$ implies $-x_2 + 3x_1 -2x_3 \geq 1$. However, adding the two yields $-x_3 - x_2 \geq 1$, a contradiction. Thus, $(3,0) \not\sqsubseteq_{\rho(x_\alpha)} (0,2)$. We can similarly compute that there are no nontrivial preorder relations under $\sqsubseteq_{\rho(x_\alpha)}$. This means we must have 
	\begin{alignat*}{4}
		- 3x_1 	& + x_3 & < 0&,\\
		- 3x_1 	& + x_3	& < 0&,\\
		   x_1 	& -2x_3 & < 1&
	\end{alignat*}
	all hold true, coming from $(3,0) \not\sqsubseteq_{\rho(x_\alpha)} (0,2)$, $(3,0) \not\sqsubseteq_{\rho(x_\alpha)} (2,1)$, and $(0,2) \not\sqsubseteq_{\rho(x_\alpha)} (2,1)$, respectively.
	
	Now let $(x_0, x_1, x_2, x_3, x_4)$ be in the interior of $F$ and cominimal with $(0,1,2,1,2)$. By Theorem \ref{Thm:Main}, we have that $L(s+5) = L(s) + 1$ for all $s \in \rho(x_\alpha)$ if and only if $-x_3 + 3x_1 \geq 3 - 0 - 1$ and $-x_0 + 2x_1 + x_3 \geq 2$. Simplifying these two inequalities yields 
	\[
		-x_3 + 3x_1 \geq 2 \quad\text{and}\quad 2x_1 + x_3 \geq 2.
	\]
	We have seen that $x_1, x_3 \geq 0$, so the only way for $2x_1 + x_3 < 2$ to happen is if $x_1 = 0$ and $x_3 \in \{0,1\}$. Since $(x_0, x_1, x_2, x_3, x_4) \in P_{\Z, 5, \Z/5\Z, \{0,1,2,3,4\}}$, we have that $x_1 + x_1 + x_1 \geq x_3$ and therefore $x_1 = 0$ implies $x_3 = 0$. Now, $x_1 = 0$ implies that $x_1 \oplus x_1 \oplus x_1 = x_3$ and therefore $x_1 \preceq x_3$. This means that if $(x_0, x_1, x_2, x_3, x_4)$ be in the interior of $F$, we automatically have $2x_1 + x_3 \geq 2$. Thus, $L(s+5) = L(s) + 1$ for all $s \in \rho(x_\alpha)$ if and only if $-x_3 + 3x_1 \geq 2$. Furthermore, we have $3x_1 > x_3$ on the interior of $F$, so we must have $-x_3 + 3x_1 = 1$ for there to exist $s \in \rho(x_\alpha)$ such that $L(s+5) \neq L(s) + 1$. 
	
	We can conclude if $(x_0, x_1, x_2, x_3, x_4)$ is an integer point in the interior of $F$ and $-x_3 + 3x_1 = 1$, there exists $s \in \rho(x_\alpha)$ such that $L(s+5) \neq L(s) + 1$. This is because $-x_3 + 3x_1 = 1$ implies $-3x_1 + x_3 < 0$ and $x_1 -2x_3 = -5x_1 + 2 < 1$ so $(x_0, x_1, x_2, x_3, x_4)$ is cominimal with $(0, 1, 2, 1, 2)$. 
	
	Now we look at some points cominimal with $(0, 1, 2, 1, 2)$ and the corresponding semigroups. 
	
	\begin{enumerate}
		\item The point $(x_0, x_1, x_2, x_3, x_4) = (0, 1, 2, 1, 2)$ corresponds to $S = \langle 5, 6, 8\rangle$. We know that $(0, 1, 2, 1, 2)$ and $(0, 1, 2, 1, 2)$ are cominimal. Since $-x_3 + 3x_1 = 2 \neq 1$, we know that $L(s + 5) = L(s) + 1$ for all $s \in S$.
		
		\item The point $(x_0, x_1, x_2, x_3, x_4) = (0, 11, 22, 32, 43)$ is in the interior of $F$. This point was obtained starting with $x_1 = 11$ and $x_3 = 32$ and using the relations $2x_1 = x_2$ and $x_1 + x_3 = x_4$. We see that $-x_3 + 3x_1 = 1$ in this case so $(0, 11, 22, 32, 43)$ and $(0, 1, 2, 1, 2)$ are cominimal and $L(s + 5) \neq L(s) + 1$ for some $s$ in the corresponding semigroup, which is $\langle 5, 56, 163 \rangle$. 
		
		\item The point $(x_0, x_1, x_2, x_3, x_4) = (0, 3, 6, 2, 5)$ is in the interior of $F$ and can be obtained by setting $x_1 = 3$ and $x_3 = 2$. We check that $-3x_1 + x_3 < 0 $ and $x_1 - 2x_3 < 1$ so $(0,3,6,2,5)$ and $(0,1,2,1,2)$ are cominimial. Here, we have $-x_3 + 3x_1 \neq 1$ so $L(s + 5) = L(s) + 1$ for all $s$ in the corresponding semigroup, which is $\langle 5, 13, 16 \rangle$. 
		
		\item The point $(x_0, x_1, x_2, x_3, x_4) = (0, 3, 6, 8, 11)$ is in the interior of $F$, determined by $x_1 = 3$ and $x_3 = 8$. In this case, we have $-x_3 + 3x_1 = 1$, so $L(s + 5) \neq L(s) + 1$ for some $s$ in the semigroup $\langle 5, 16, 43 \rangle$.  
	\end{enumerate}

\end{example}

\begin{remark}
	For a numerical semigroup, if we let $m$ be the largest atom, there is only one integer point on the face of a Kunz polytope that corresponds to a numerical semigroup with $m$ being the largest atom. This is because $x_\alpha = 0$ (if we take $r_{n+m\Z} = n \in \{0, 1, \dots, m-1\}$) for all $\alpha \in \mathcal{A}(\mathcal{P}^\infty)$ and this determines $x_\alpha$ for all $\alpha$. 

\end{remark}

\bibliographystyle{amsalpha}
\bibliography{references}
\end{document}